\documentclass[12pt]{amsart}

\usepackage{amsmath, amssymb}
\usepackage{enumitem}
\usepackage{booktabs}

\usepackage{multirow}

\usepackage{tabularx}

\usepackage{graphicx}

\usepackage{MnSymbol}
\newtheorem{theorem}{Theorem}[section]
\newtheorem{lemma}[theorem]{Lemma}
\newtheorem{prop}[theorem]{Proposition}

\newtheorem{corollary}[theorem]{Corollary}

\theoremstyle{definition}

\theoremstyle{remark}
\newtheorem{remark}[theorem]{Remark}

\newcommand{\rd}{{\mathbb R^d}}
\newcommand{\rr}{{\mathbb R}}

\newcommand{\N}{{\mathbb N}}

\newcommand{\p}{{\mathcal P}}
\renewcommand{\d}{\mathrm{d}}
\newcommand{\lex}{<_{\text{lex}}}
\newcommand{\EE}{\mathbb{E}}
\newcommand{\PP}{\mathbb{P}}
\newcommand{\ZZ}{\mathbb{Z}}
\newcommand{\RR}{\mathbb{R}}
\newcommand{\NN}{\mathbb{N}}

\usepackage{comment}
\usepackage{tikz}
\usepackage{float}
\usepackage{relsize}
\usepackage{enumitem}
\usepackage[T1]{fontenc}
\usepackage[english]{babel}
\usepackage{xcolor}
\usepackage{graphicx}
\usepackage{cite}
\usepackage{mathtools}

\date{\today}

\begin{document}

\sloppy

\title[{POT in Scale-Free Random Graphs}]{Peaks Over Threshold in Scale-Free Random Graphs}

\author{Arnaud Rousselle}
\address{Arnaud Rousselle, Université Bourgogne Europe, CNRS, IMB UMR 5584, F-21000 Dijon, France}
\email{arnaud.rousselle@u-bourgogne.fr}

\author{Ercan S\"onmez}
\address{Ercan S\"onmez, Faculty of Mathematics, Bochum, Germany}
\email{ercan.soenmez\@@{}rub.de}

\begin{abstract}
We explore extreme value phenomena in spatial {scale-free} random graphs in a continuum setting based on a homogeneous Poisson point process in $\RR^d$.
Vertices carry i.i.d.\ weights $(W_x)$ and, conditionally on the vertex set and the weights, edges are present independently with probability
$p_{xy}=1-\exp\{-\lambda W_xW_y/|x-y|^\alpha\}$.
Assuming Pareto-type weight tails with index $\beta>0$ and working in parameter ranges where degrees are almost surely finite, we study extremes and peaks over thresholds (POT) of edge lengths in a growing observation window. Our focus is the precise impact of the presence of (large) weights on edge lengths, captured through explicit scaling regimes and conditional POT limit theorems. Our main results identify a three-phase behavior governed by the weight-tail parameter $\beta$.
We both deduce Fr\'echet-type limits for the maximum edge length itself and we reveal POT structures under a hub conditioning by proving a POT limit theorem. In the finite-mean regime $\beta>1$, the leading scaling agrees with the unweighted model up to a constant. By contrast, for $\beta\le 1$ the weights have a macroscopic effect on extreme edge lengths: for $\beta< 1$ the scaling changes, and the borderline case $\beta=1$ exhibits additional logarithmic corrections.
%We show convergence to a non-degenerate limit $G$ whose tail is asymptotically equivalent to a generalized Pareto tail, with parameter determined by the corresponding extreme-value index.
The proofs combine Stein-type Poisson approximation via a Palm--coupling approach with a refined treatment of the dependence created by the conditioning event.
\end{abstract}

\keywords{Scale-free random graphs, extreme value theory, peaks over threshold, phase transition.}
\subjclass{Primary: 60G70, 05C80, ; Secondary: 60F05, 05C82.}

\maketitle

\baselineskip=18pt
\sloppy

\section{Introduction}\label{sec:intro}

\subsection{Motivation}\label{subsec:motivation}

Spatial random graphs provide a natural framework for modeling networks in which geometry constrains the formation of links.
A prominent example is {long-range percolation} and its continuum analogue, the {random connection model}, where the likelihood of an edge between two vertices decreases with their Euclidean separation.
Such distance-dependent connections offer a parsimonious mechanism for the \emph{small-world phenomenon}: the presence of sufficiently many long edges can drastically reduce graph distances between remote vertices.

A second stylized feature frequently observed in real networks is strong degree heterogeneity, often summarized by the term \emph{scale-free}.
A mathematically tractable way to incorporate hubs into spatial long-range models is to assign i.i.d.\ vertex weights and let the edge probabilities depend both on weights and distance; this leads to {scale-free percolation} on $\ZZ^d$ and to heterogeneous random connection models on $\RR^d$.
These models have been introduced and studied with a focus on degree distributions, percolation properties, and graph distances; see, for example, the lattice setting in \cite{deijfen2013scale} and its continuum counterpart in \cite{deprez2019scale}.

In practice, we consider scale-free graphs in a continuum setting. The vertex set is given by a stationary Poisson point process in $\RR^d$.
Each vertex $x$ carries an i.i.d.\ weight $W_x\ge 0$ (independent of the locations).
Conditionally on the vertex set and the weights, edges are present independently, and the connection probability between two distinct vertices
$x$ and $y$ is taken of the form
\begin{equation}\label{eq:conn-prob-intro}
p_{xy}=1-\exp\!\left\{-\lambda\,\frac{W_xW_y}{|x-y|^\alpha}\right\},
\qquad \lambda>0,\ \alpha>0,
\end{equation}
where $|\cdot|$ denotes the Euclidean norm.
The parameter $\alpha$ controls the long-range decay in the distance, while $\lambda$ acts as an overall intensity parameter for the appearance of the edges.
For concreteness and to obtain explicit asymptotics, we will assume Pareto-type tails for the weights,
that is, $\PP(W_x>w)=w^{-\beta}$ for $w\ge 1$ with some $\beta>0$.

From a theoretical viewpoint, these models also exhibit a rich variety of {extreme value} phenomena.
In earlier work \cite{rs1, rs2}, we initiated a systematic extreme value analysis for {unweighted} discrete and continuous long-range percolation by studying the maximal length of edges with at least one endpoint in a growing observation window.
In the heavy-tailed regime for the connection function, the normalized maximum converges to the Fr\'echet law, except for a critical behavior at the borderline exponent $\alpha=2d$, leading to an unexpected non-standard limiting distribution.
This suggests that long-range spatial graphs are a natural source of further extreme value structures beyond classical block-maxima limits.

The present paper continues this line of research in two directions.
First, we investigate how vertex weights alter the scale of the longest edges in weighted long-range models.
Second, we address a peaks over threshold (POT) question: conditioning on the presence of unusually large weights (hubs), how are extreme edge lengths affected?
Our results uncover a sharp dependence on the tail parameter of the weights, including a borderline regime in which logarithmic corrections emerge.

\subsection{Main results}\label{subsec:main-results}

We denote the corresponding random graph by $(V,E)$.
We consider an observation window
\[
B_n=[-n,n]^d,
\]
and we study the longest edge with at least one endpoint in $B_n$,
\[
e_n^*:=\max_{\{x,y\}\in E_{n}} |x-y|,
\qquad
E_{n}:=\bigl\{\{x,y\}\in E:\ x\in V\cap B_n\bigr\}.
\]
A standard route to extreme value limits is via exceedances.
For thresholds $r_n$ depending on a level parameter $r>0$, we consider the number of {exceeding} edges touching $B_n$,
\begin{align}
F(n,r)
&:= \sum_{x\in V\cap B_n}\sum_{y\in V\cap B_n^c}
\mathbf{1}_{\{|x-y|>r_n,\, x\leftrightarrow y\}}
+ \frac12\sum_{x\in V\cap B_n}\sum_{y\in V\cap B_n}
\mathbf{1}_{\{|x-y|>r_n,\, x\leftrightarrow y\}},
\label{eq:Fnr-intro}
\end{align}
where $\{x\leftrightarrow y\}$ denotes the event that $x$ and $y$ are connected by an edge. Then $\{e_n^*\le r_n\}$ is equivalent to $\{F(n,r)=0\}$.
Our analysis identifies thresholds $r_n$ such that $F(n,r)$ converges to a Poisson law with explicit mean measure $\Lambda(r)$; this yields Fr\'echet-type limits for $e_n^*$.

\medskip\noindent

A complementary viewpoint to such maxima, important in extreme value theory, is provided by \emph{peaks over threshold} (POT):
rather than considering only the maximum, one studies the distributional structure of exceedances above a high level.
In the heavy-tailed (Fr\'echet) domain, POT limits are governed by generalized Pareto tails.
Concretely, if $X\ge 0$ has a regularly varying tail with index $-\kappa$ for some $\kappa>0$, that is,
\[
\PP(X>x)=x^{-\kappa}L(x), \qquad x\to\infty,
\]
with $L$ slowly varying, then
\begin{equation*}%\label{eq:classical-POT}
\PP\!\left(\frac{X-u}{u}>t \,\bigg|\, X>u\right)\ \longrightarrow\ (1+t)^{-\kappa},
\qquad t\ge 0,
\end{equation*}
as $u\to\infty$.
Equivalently, the conditional excess distribution above a high threshold converges to a generalized Pareto law with tail index (and scale parameter) $\kappa$. In this spirit, we study a POT question in which extreme edge lengths are observed under a rare-event conditioning that enforces the presence of a hub.
Let
\[
W_n^*:=\max_{x\in V\cap B_n} W_x,
\]
and consider conditional probabilities of the form
\[
\PP\!\left(e_n^*>t_n \,\big|\, W_n^*>d_n\right),
\]
for suitable threshold sequences $d_n$ and $t_n$.
In this way, the threshold is imposed through an extremal weight event, and we ask how the resulting conditional law of extreme edge lengths compares to the classical POT paradigm.

Before stating our main results, we recall a basic well-posedness property of the model.
In both the lattice and the continuum setting, the degree of every vertex is almost surely finite provided that
\begin{equation}\label{eq:finite-degree-cond-intro}
\min\{\alpha,\alpha\beta\}>d.
\end{equation}
In the lattice model this follows from \cite[Theorem~3.1]{deijfen2013scale}, while the corresponding statement
for the continuum model is proved in \cite[Theorem~3.1]{deprez2019scale}.
Throughout the sequel we therefore work under \eqref{eq:finite-degree-cond-intro}.

A first novel feature of the weighted model is a sharp three-phase behavior governed by the weight-tail parameter $\beta$.
In the finite-mean regime $\beta>1$, the weights do not change the leading order of the longest-edge scale compared to the unweighted model, except through an explicit multiplicative constant.
By contrast, when $\beta\le 1$ (so that $\EE[W]=\infty$), the weights have a genuine effect on extreme edge lengths:
for $\beta\leq 1$ the scaling changes, and the borderline case $\beta=1$ yields a distinctive regime with additional logarithmic corrections.
The following theorem makes this trichotomy precise in terms of Poisson limits for exceedances and the resulting Fr\'echet-type limits for the maximum.

\begin{theorem}[Unconditional exceedances and maximum edge length]\label{thm:uncond-max}
Consider the continuum model in which $V$ is a homogeneous Poisson point process on $\RR^d$ with unit intensity, and let the connection probabilities be given by \eqref{eq:conn-prob-intro}.
For $n\in\NN$ set $B_n=[-n,n]^d$ and define
\[
e_n^*:=\max_{\{x,y\}\in E_{n}} |x-y|,
\qquad
E_{n}:=\bigl\{\{x,y\}\in E:\ x\in V\cap B_n\bigr\},
\]
as well as the exceedance count $F(n,r)$ from \eqref{eq:Fnr-intro}.
Assume that the weight distribution satisfies $\PP(W>w)=w^{-\beta}$ for $w\ge 1$, with tail index $\beta>0$, and that the
finite-degree condition \eqref{eq:finite-degree-cond-intro} holds.

\smallskip
For $r>0$ set $r_n=c_n r$ and define
\[
\theta=
\begin{cases}
\alpha\beta-d, & 0<\beta<1,\\
\alpha-d, & \beta\ge 1.
\end{cases}
\]
Then, in each of the regimes \textup{(a)}--\textup{(c)} below, with the corresponding choice of $c_n$, we have
\begin{equation*}\label{eq:pois-frechet-unified}
F(n,r)\ \xrightarrow[n\to\infty]{d}\ \mathrm{Poisson}\!\bigl(r^{-\theta}\bigr),
\qquad
\PP\!\left(c_n^{-1}e_n^*\le r\right)\ \longrightarrow\ \exp\!\bigl(-r^{-\theta}\bigr),
\qquad r>0,
\end{equation*}
that is, $c_n^{-1}e_n^*$ converges in distribution to a Fr\'echet random variable with parameter $\theta$.

\smallskip
\noindent\textup{(a) {Infinite-mean regime} $(0<\beta<1)$.}
Assume $\alpha\beta\in(d,2d]$ and set
\[
c_n=C_1\,n^{\frac{d}{\alpha\beta-d}}(\log n)^{\frac{1}{\alpha\beta-d}}.
\]

\smallskip
\noindent\textup{(b) {Borderline regime} $(\beta=1)$.}
Assume $\alpha\in(d,2d]$ and set
\[
c_n=C_2\,n^{\frac{d}{\alpha-d}}(\log n)^{\frac{2}{\alpha-d}}.
\]

\smallskip
\noindent\textup{(c) {Finite-mean regime} $(\beta>1)$.}
Assume $\alpha\in(d,2d)$ and set
\[
c_n=C_3\,n^{\frac{d}{\alpha-d}}.
\]
Here $C_1, C_2, C_3\in(0,\infty)$ denote explicit constants independent of $n$ (specified in
Corollary~\ref{cor:cn-choice}).
\end{theorem}

\begin{remark}\label{rem:constants-discrete-continuum}
The constants $C_1, C_2, C_3$ in the three regimes above depend on $(d,\alpha,\lambda)$ and on the weight law; for the sake of readability their explicit forms are given in Corollary~\ref{cor:cn-choice}.
Moreover, the borderline exponent $\alpha=2d$ when $\beta>1$ leads to a non-standard limit law;
similarly as in the unweighted model. In contrast, for $\beta\le 1$ the regimes treated in Theorem~\ref{thm:uncond-max} explicitly include $\alpha\beta=2d$ and no additional critical phenomenon occurs. The critical case $\alpha=2d$ for $\beta>1$ yields a limit law of the same nature as in \cite{rs1}. We chose not pursue it here, since the phenomenon is not new to the weighted setting. Moreover, in the exceedance count $F(n,r)$ in \eqref{eq:Fnr-intro}, the decomposition into edges leaving $B_n$ and edges with both endpoints in $B_n$ (including the factor $1/2$ in the latter term) accounts for boundary effects and prevents double counting. Another appraoch, often technically more convenient, is obtained by counting only edges whose midpoints lie in $B_n$.
\end{remark}

\begin{remark}
One may also consider the modified model in which, conditionally on $(V,(W_x)_{x\in V})$, edges are present independently with probability
\[
p_{xy}=\exp\Big\{-\,\frac{|x-y|}{\lambda W_xW_y}\Big\},\qquad x\neq y.
\]
In the unweighted case $W\equiv 1$, the connection probability decays exponentially in $|x-y|$, and the longest edge has a Gumbel-type behaviour (with a logarithmic order of magnitude). By contrast, once heavy-tailed weights are introduced, the mixture over $(W_xW_y)$ produces a polynomial tail for long connections, and our methods yield a Fr\'echet-type limit theorem for the maximum edge length in this modified model for {every} $\beta>0$. In particular, the presence of weights changes the extreme-value universality class and the scale of the longest edges even in the finite-mean regime $\beta>1$.
\end{remark}

% --- Peaks over threshold: conditional extremes under a hub event ---

\begin{theorem}[Peaks over threshold under a hub conditioning]\label{thm:POT-hub}
Let the assumptions of Theorem~\ref{thm:uncond-max} hold and define
\[
W_n^*:=\max_{x\in V\cap B_n} W_x .
\]
Fix $t>0$ and let $c_n\to\infty$. Set
\[
\theta=
\begin{cases}
\alpha\beta-d, & 0<\beta<1,\\
\alpha-d, & \beta\ge 1,
\end{cases}
\qquad\text{and}\qquad
t_n=t_n(t):=c_n^{1/\theta}\Bigl(1+\frac{t}{\theta}\Bigr).
\]
Assume further that
\begin{equation}\label{eq:tn-growth-POT-thm}
\frac{t_n}{\max\{n,\; n^{d/\theta} (\log n)^{\frac{2}{\theta}}\}}\ \longrightarrow\ \infty.
\end{equation}
Define the hub level $d_n$ by
\[
d_n:=
\begin{cases}
 \beta^{\!1/\beta} c_n^{\,\frac{1}{\beta}} 
\left({K_\beta\,\log\!\big(K_\beta\,c_n^{\frac{d}{\theta}}\big)}\right)^{-\!1/\beta}, & 0<\beta<1,\\[1mm]
K_1^{-1}\,c_n/\log^2 (K_1 c_n^{\frac{d}{\theta}}), & \beta=1,\\[1mm]
K_\beta^{-1}\,c_n, & \beta>1,
\end{cases}
\]
where the explicit constants $K_\beta = K_{\alpha,\beta,d,\lambda}\in(0,\infty)$ are given as in equation \eqref{eq:kappa-values}. 
%\[
%\kappa_\beta:=
%\begin{cases}
%\bigl(\Gamma(1-\beta)^2\lambda^\beta\,\dfrac{d\omega_d}{\alpha\beta-d}\bigr)^{1/\beta}, & 0<\beta<1,\\[2mm]
%\lambda\,\dfrac{d\omega_d}{\alpha-d}\,\dfrac{d}{\alpha-d}, & \beta=1,\\[2mm]
%\lambda\,\dfrac{d\omega_d}{\alpha-d}\,\dfrac{\beta^2}{(\beta-1)^2}, & \beta>1.
%\end{cases}
%\]
Moreover, denote by $\operatorname{GPD}$ the tail function of the generalized Pareto distribution with index $\theta$ given by
$$ \operatorname{GPD} (t) = \Bigl(1+\frac{t}{\theta}\Bigr)^{-\theta}, \quad t \in (0. \infty).$$
\smallskip
Then, as $n\to\infty$,
\begin{equation*}\label{eq:POT-limit-thm}
\PP\!\left(\frac{e_n^* - c_n^{1/\theta} }{\theta^{-1}c_n^{1/\theta}}>t\,\Big|\, W_n^*>d_n\right)=\PP\!\left(e_n^*>t_n(t)\,\Big|\, W_n^*>d_n\right)\ \longrightarrow\ \operatorname{GPD}(t),
\qquad t>0. 
\end{equation*}

That is, conditional on the hub event $\{W_n^*>d_n\}$, the normalized excess of the maximum edge length converges in distribution to a generalized Pareto law with tail index $\theta$.
\end{theorem}

We now briefly comment on the methods used in the paper to prove our theorems.
The Poisson approximation for the exceedance counts is established as follows.
%In the lattice model we use Stein's method based on a Poisson coupling argument due to \cite{barbour}.
First we develop a Palm--coupling variant of Stein's method for Poisson functionals, extending ideas from \cite{Penrose2}; this approximation scheme may be of independent interest.
Then, an essential input is a precise asymptotic evaluation of the mean number of exceedances in the relevant scaling regimes.

The conditional peaks over threshold results are more delicate due to the dependence created by conditioning on an extreme weight event.
For the threshold sequences considered here, an important observation we make is that the event $\{W_n^*>d_n\}$ is typically realized by a unique exceptional vertex (a single hub), and the resulting asymptotics are governed by edges incident to this vertex.

We choose not to treat the discrete model with $V=\mathbb{Z}^d$. The same statements hold in this setting (with slightly modified constants), and the proofs follow the same structure. The main difference is that on $\mathbb{Z}^d$ one can rely on a Poisson coupling argument due to \cite{barbour}. Since this makes the lattice case technically less demanding, and since it was already a main theme of our earlier work, we put attention to the continuum.

\medskip
The remainder of the paper is organized as follows.
In Section~\ref{sec:intermediate} we collect intermediate results that form key parts of the proofs. %and can be established largely independently of whether one works on the lattice or in continuum space.
This section also contains additional results on the maximum edge length, and further related POT results. Section~\ref{sec:continuum} is devoted to the proof of Theorem \ref{thm:uncond-max}; in particular, we state the Poisson approximation theorem based on a Palm--coupling approach and derive the corresponding exceedance limits. Finally, Section~\ref{proofthm2} is devoted to the proof of Theorem \ref{thm:POT-hub}

\section{Intermediate and further results} \label{sec:intermediate}

This section collects auxiliary statements that will also be used in the proofs of the main theorems. It is convenient to formulate some statements for the longest edge incident to a distinguished vertex. We use the symbol $\PP_0$ to denote the law of the model viewed from a typical vertex:
this means $\PP_0$ refers to the Palm version of the underlying marked Poisson process. Under $\PP_0$ there is almost surely a vertex at the origin, and we may speak of quantities such as the longest edge incident to $0$.

We start with the following statement. From now on we define the {longest edge incident to the origin} by
\[
e_0^* \,:=\, \sup\bigl\{|y|:\ y\in V\setminus\{0\},\ 0\leftrightarrow y\bigr\}.
\]

Before stating the next peaks over threshold result, we fix a notion of tail equivalence.
For functions $F$ and $H$ on $(0,\infty)$ we write $\bar F\sim \bar H$ as $t\to\infty$ if
\[
\lim_{t\to \infty}\frac{1-F(t)}{1-H(t)}\ = 1,
\]
where $\bar F(t)=1-F(t)$ and $\bar H(t)=1-H(t)$.

\begin{lemma}\label{lem:POT-rooted-unrooted}
Let $(c_n)_{n\in\NN}$ be an increasing sequence and set
\[
d_n:=
\begin{cases}
\kappa_\beta^{-1}\,c_n^{1/\beta}, & 0<\beta<1,\\[1mm]
\kappa_1^{-1}\,c_n/\log c_n, & \beta=1,\\[1mm]
\kappa_\beta^{-1}\,c_n, & \beta>1,
\end{cases}
\qquad
\theta=
\begin{cases}
\alpha\beta-d, & \beta<1,\\
\alpha-d, & \beta\ge 1,
\end{cases}
\qquad
t_n=t_n(t):=c_n^{1/\theta}\Bigl(1+\frac{t}{\theta}\Bigr),
\]
for $t>0$, where $\kappa_\beta\in(0,\infty)$ is an explicit constant independent of $n$ (see Lemma \ref{lem:In-limit-allbeta-gpd}). Then there exists a function $G=G_\beta$ on $(0,\infty)$ such that, as $n\to\infty$,
\begin{equation*}
\PP_0\!\left(e_0^*>t_n(t)\,\Big|\, W_0>d_n\right)\ \longrightarrow\ G(t),
\qquad t>0.
\end{equation*}
Assume moreover that \eqref{eq:tn-growth-POT-thm} holds. Then, as $n\to\infty$,
\begin{equation*}
\PP_0\!\left(e_n^*>t_n(t)\,\Big|\, W_0>d_n\right)\ \longrightarrow\ G(t),
\qquad t>0.
\end{equation*}
Moreover, $G$ has generalized Pareto tail with index $\theta$, that is its tail admits the explicit representation
\begin{equation*}
\bar G(t)\ \sim\ \Bigl(1+\frac{t}{\theta}\Bigr)^{-\theta}\,,
\qquad t\to\infty, \quad \beta \neq 1,
\end{equation*}
and
\begin{equation*}
\bar G(t)\ \sim\ \Bigl(1+\frac{t}{\theta}\Bigr)^{-\theta}\,L(t),
\qquad t\to\infty, \quad \beta = 1,
\end{equation*}
where $L$ is slowly varying and can be chosen as
\[
L(t)=
\theta\,\log(1+t).
\]
\end{lemma}

Next we state the following.

\begin{lemma}\label{lem:rooted-exceed-asymptotics}
Assume that \eqref{eq:tn-growth-POT-thm} holds. Let $d_n\to\infty$ satisfy
\[
d_n =
\begin{cases}
o(t_n^{\alpha\beta}), & \beta<1,\\[1mm]
o\!\bigl(t_n^{\alpha}/\log t_n\bigr), & \beta=1,\\[1mm]
o(t_n^{\alpha}), & \beta>1.
\end{cases}
\]
Then, as $n\to\infty$,
\begin{equation*}
\PP_0\!\left(e_0^*>t_n \,\Big|\, W_n^*>d_n\right)
\ \sim\
\kappa
\begin{cases}
t_n^{\,d-\alpha\beta} \log t_n^\alpha, & \beta<1,\\[1mm]
t_n^{\,d-\alpha}\,( \log t_n^\alpha)^2, & \beta=1,\\[1mm]
t_n^{\,d-\alpha}, & \beta>1,
\end{cases}
\end{equation*}
where $\kappa\in(0,\infty)$ is an explicit constant independent of $n$.
\end{lemma}

\begin{remark} \label{re:intuitionmaxweights}
    In the following we consider the regime $\frac{d_n^\beta}{n^d} \to \infty$. If $Y$ denotes a random variable with binomial distribution $\operatorname{Bin} (2^dn^d, d_n^{-\beta})$ then its mode is given by $\lfloor (2^dn^d +1)d_n^{-\beta} \rfloor $ or $\lceil (2^d n^d + 1) d_n^{-\beta} \rceil -1$. By assumption, both expressions tend to $0$, as $n \to \infty$. Thus, heuristically the maximum weight $W_n^*$ does not exceed $d_n$ with high probability (i.e. this is an unlikely event, as typically considered in the theory of peaks over thresholds). Moreover, intuitively, we can observe that if the maximum weight exceeds $d_n$ in this regime, then with high probability there is only one exceedance. This observation will play a crucial role in the following.
\end{remark}

\begin{lemma}\label{le:jointedgesweights}
    Suppose the assumptions in Lemma \ref{lem:rooted-exceed-asymptotics} are satisfied. Consider the random variable
    $$ X_n := \sum_{x \in V \cap B_n} \sum_{y \in V \cap B^C_{t_n} (x)} \mathbf{1}_{\{ {W_x > d_n}  \}}  \mathbf{1}_{{\{x\leftrightarrow y}  \}}.$$
    Let $Z_n$ be a random variable with Poisson distribution and $\mathbb{E}[Z_n] = \mathbb{E}[X_n]$. Then
    \begin{equation*}d_{\text{TV}}\left(X_n,Z_n\right)\leq c\begin{cases}
t_n^{2(d-\alpha\beta)}
\Bigl(1+\log\!\bigl(t_n^\alpha/d_n\bigr)\Bigr)^{2},
& 0<\beta<1,\\[1mm]
t_n^{2(d-\alpha)}
\Bigl(1+\log^{2}\!\bigl(t_n^\alpha/d_n\bigr)\Bigr)^{2},
& \beta=1,\\[1mm]
t_n^{2(d-\alpha)}\,d_n^{2(1-\beta)},
& \beta>1,
\end{cases}
\end{equation*}
where $c$ is a positive constant independent of $n$. Moreover, it holds
$$ \mathbb{E}[X_n] \sim K (2n)^d\cdot
\begin{cases}
t_n^{\,d-\alpha\beta}\,\log\!\Bigl(\dfrac{t_n^\alpha}{d_n}\Bigr), & 0<\beta<1,\\[2.5mm]
t_n^{\,d-\alpha}\,\log^2\!\Bigl(\dfrac{t_n^\alpha}{d_n}\Bigr), & \beta=1,\\[2.5mm]
t_n^{\,d-\alpha}\,d_n^{\,1-\beta}, & \beta>1,
\end{cases}, $$
where $K=K_{\alpha,\beta,d,\lambda}$ is given by \eqref{eq:kappa-values}.
\end{lemma}

\section{Proof of Theorem \ref{thm:uncond-max}}\label{sec:continuum}
%\subsection{Continuous model} 
\rule{0pt}{1.2\baselineskip}
%\section{Solved}

%\subsection{Poisson approximation}\label{poissonapproximation}
%(mistake in Iyer and Jhawar : corrected)
We first extend a Poisson approximation result from \cite{Penrose2} to account for the additional complexity introduced by random weights in the scale-free model, see \cite{iyer2} for a similar idea. The proof builds on Theorem 3.1 in \cite{Penrose2}, which considers a marked Poisson point process \(\eta\) consisting of a homogeneous Poisson point process on \(\mathbb{R}^d\) with intensity \(\rho \in (0, \infty)\) and marks in a measure space \((\mathbb{M}, \mathcal{M}, m)\), where \(m\) is a diffusive probability measure. 

In order to incorporate the random weights, we modify the construction of the underlying model. The construction is given in Section \ref{ssec:construction}. %As in \cite{Penrose2}, we use the notation \(\rho(\d x) = \rho \d x\) for the intensity measure. 
Due to the additional randomness in the model, adjustments to the proof of \cite[Theorem 3.1]{Penrose2} are required to obtain the version applicable in our setting. We denote by \(S\) the space of all locally finite subsets of \(\mathbb{R}^d \times \mathbb{M}\) and by \(S_k\) the collection of subsets of \(\mathbb{R}^d \times \mathbb{M}\) of cardinality \(k\), for \(k \in \mathbb{N}\). Our Poisson approximation result reads as follows.

\begin{theorem}\label{Th:weightedstein}

Let $k\in\mathbb{N}$, $f\colon\mathbf{S}_k\times \mathbf{S} \longrightarrow \{0,1\}$ a measurable function and for $\xi\in \mathbf{S}$ set: 
\[F(\xi):=\sum_{\psi\in\mathbf{S}_k:\psi\subset\xi}f(\psi,\xi\setminus \psi).\]

Let $\eta$ be a (marked) Poisson point process with intensity $\rho\times \mathbf{m}$ in $\mathbb{R}^d\times \mathbb{M}$ and set $Z:=F(\eta)$ and $\zeta:=\mathbb{E}[Z]$. For $x_1,\dots,x_k\in\mathbb{R}^d$ with corresponding weights $W_{x_1},\dots,W_{x_k}$, respectively, define 
$$p(x_1,\dots,x_k, W_{x_1},\dots,W_{x_k}):=\mathbb{E}\left[f\left((x_1,\tau_1,W_{x_1}),\dots,(x_k,\tau_k,W_{x_k}),\eta\right) \mid W_{x_1},\dots,W_{x_k}\right]$$ 
where the $(\tau_i, W_{x_i})$ are independent random elements of $\mathbb{M}$ with common distribution $\mathbf{m}$. 

Suppose that for almost every $\mathbf{x}=(x_1,\dots,x_k)\in \left(\mathbb{R}^d\right)^k$ with weight vector $\mathbf{W}= (W_{x_1},\dots,W_{x_k})$ such that $p(x_1,\dots,x_k, W_{x_1},\dots,W_{x_k})>0$ we can find coupled random variables $U_{\mathbf{x},\mathbf{W}}$ and $V_{\mathbf{x},\mathbf{W}}$ satisfying:
\begin{enumerate}
\item[(i)] $Z\overset{d}{=}U_{\mathbf{x},\mathbf{W}}$,
\item [(ii)]\label{Pen2} Conditional on $\mathbf{W}$: the random variable $F\left(\eta\cup\overset{k}{\underset{i=1}{\bigcup}}\{(x_i,\tau_i, W_{x_i})\}\right)$  conditioned on $f\left(\overset{k}{\underset{i=1}{\bigcup}}\{(x_i,\tau_i)\},\eta\right)=1$ has the same distribution as $1+V_{\mathbf{x},\mathbf{W}}$,
\item[(iii)] $\mathbb{E}\left[\mathbb{E} [\vert U_{\mathbf{x},\mathbf{W}}-V_{\mathbf{x},\mathbf{W}}\vert ]p(\mathbf{x}, \mathbf{W})\right]\leq w(\mathbf{x})$ where $w$ is a measurable function.
\end{enumerate}

Let $P(\zeta)$ be a mean $\zeta$ Poisson random variable. Then

\begin{equation*}d_{\text{TV}}\left(Z,P(\zeta)\right)\leq\frac{\min(1,\zeta^{-1})}{k!}\int_{(\mathbb{R}^d)^k}w(\mathbf{x})\operatorname{d} \mathbf{x}\label{th:Pdtv}\end{equation*}
and 
\begin{equation*}d_{\text{W}}\left(Z,P(\zeta)\right)\leq\frac{3\min(1,\zeta^{-\frac{1}{2}})}{k!}\int_{(\mathbb{R}^d)^k}w(\mathbf{x})\operatorname{d} \mathbf{x}.\label{th:PdW}\end{equation*}

\end{theorem}

\begin{proof}
We employ Stein’s method based on Palm coupling, with modifications to account for the presence of random weights. For notational simplicity we write $\d \mathbf{x} = \rho^k (\d \mathbf{x})$ in the following. Moreover, we abbreviate $p(\mathbf{x}) = p(\mathbf{x}, \mathbf{W})$.

Let \( h\colon \mathbb{N}_0 \to \mathbb{R} \) be a bounded function. Then,
\begin{align*}
   & \mathbb{E} [Z h(Z)] = \mathbb{E} \left[ \sum_{\psi\subset\eta: |\psi|=k} f(\psi, \eta \setminus \psi) h(F(\eta))\right] \\
&= \frac{1}{k!} \int \mathbb{E} \left[ f\left(\{(x_1, \tau_1, W_{x_1}), \dots, (x_k, \tau_k, W_{x_k})\}, \eta\right) h \left( F \left( \eta \cup \bigcup_{i=1}^{k} \{(x_i, \tau_i, W_{x_i})\} \right) \right) \right] \d \mathbf{x} \\
&= \frac{1}{k!} \int \mathbb{E} \left[ \mathbb{E} \left[ f\left(\{(x_1, \tau_1, W_{x_1}), \dots, (x_k, \tau_k, W_{x_k})\}, \eta\right) h \left( F \left( \eta \cup \bigcup_{i=1}^{k} \{(x_i, \tau_i, W_{x_i})\} \right) \right) \Big\vert \mathbf{W} \right] \right] \d \mathbf{x} \\
&= \frac{1}{k!} \int \mathbb{E} \left[ \mathbb{E} \left[ h \left( F \left( \eta \cup \bigcup_{i=1}^{k} \{(x_i, \tau_i, W_{x_i})\} \right) \right) \Bigg\vert \mathbf{W}  \right]  \Big\vert f \left( \bigcup_{i=1}^{k} \{(x_i, \tau_i)\}, \eta \right) = 1  \right] \mathbb{E} [p(\mathbf{x})] \d \mathbf{x}\\
& = \frac{1}{k!} \int \mathbb{E} \left[ \mathbb{E} \left[ h \left(  1+V_{\mathbf{x},\mathbf{W}} \right) \vert \mathbf{W} \right]    \right] \mathbb{E} [p(\mathbf{x})] \d \mathbf{x} \\
& = \frac{1}{k!} \int \mathbb{E} \left[  h \left(  1+V_{\mathbf{x},\mathbf{W}} \right)  \right] \mathbb{E} [p(\mathbf{x})] \d \mathbf{x} .
\end{align*}
Moreover, note that we have
\begin{align*}
    \zeta & = \mathbb{E}[F(\eta)] = \frac{1}{k!} \int\mathbb{E}[p(\mathbf{x})] \d \mathbf{x} ,
\end{align*}
so that
\begin{align*}
    \zeta \mathbb{E} [ h(Z+1)] - \mathbb{E} [Z h(Z)] & = \frac{1}{k!} \int \mathbb{E} \left[ \left(\mathbb{E} [ h(Z+1)] -  \mathbb{E} \left[h \left(  1+V_{\mathbf{x},\mathbf{W}} \right)  \right] \right)  p(\mathbf{x}) \right]  \d \mathbf{x} \\
    & = \frac{1}{k!} \int \mathbb{E} \left[ \left(\mathbb{E} [ h(1+U_{\mathbf{x},\mathbf{W}})] -  \mathbb{E} \left[h \left(  1+V_{\mathbf{x},\mathbf{W}} \right)  \right] \right)  p(\mathbf{x}) \right]  \d \mathbf{x} .
\end{align*}
Thus, we get
\begin{align*}
    \left| \mathbb{E} [\zeta h(Z + 1) - Z h(Z )] \right| &\leq \frac{1}{k!} \int \mathbb{E} \left[ \mathbb{E} [ \left |h(1+U_{\mathbf{x},\mathbf{W}}) -h \left(  1+V_{\mathbf{x},\mathbf{W}} \right)  \right| ]   p(\mathbf{x}) \right]  \d \mathbf{x} .
\end{align*}
Since it holds
\[
|h(i) - h(j)| \leq \|\Delta h\|_{\infty} \cdot |i - j|, \quad \text{for } i, j \in \mathbb{N}_0,
\]
we obtain
\begin{align*}
    \mathbb{E} [ \left |h(1+U_{\mathbf{x},\mathbf{W}}) -h \left(  1+V_{\mathbf{x},\mathbf{W}} \right) ) \right| ]   
    &\leq \|\Delta h\|_{\infty} \mathbb{E}[ |U_{\mathbf{x},\mathbf{W}} - V_{\mathbf{x},\mathbf{W}}| ]
\end{align*}
and from this
\begin{align*}
    \left| \mathbb{E} [\zeta h(Z + 1) - Z h(Z )] \right| &\leq \frac{1}{k!} \|\Delta h\|_{\infty}\int \mathbb{E} \left[ \mathbb{E}[ |U_{\mathbf{x},\mathbf{W}} - V_{\mathbf{x},\mathbf{W}}| ]  p(\mathbf{x}) \right]  \d \mathbf{x} \\
    & \leq \frac{1}{k!} \|\Delta h\|_{\infty}\int w(\mathbf{x} )  \d \mathbf{x}.
\end{align*}
where we used the assumption in (iii). Now we conclude using \cite[Lemma 1.1.1]{barbour}. Indeed, given \( A \subset \mathbb{N}_0 \), set \( g = \mathbf{1}_A \) and specify \( h\) such that \( h(0) = 0 \) and
\begin{align*}
    \zeta h(i + 1) - i h(i) = g(i) - \mathbb{E} [G(P(\zeta))], \quad i \in \mathbb{N}_0. %\label{eq:h_equation}
\end{align*}
Then \( h \) is bounded with \( \|\Delta h\|_{\infty} \leq 1 \wedge \zeta^{-1} \), and hence
\begin{align*}
    \left| \mathbb{P}[Z \in A] - \mathbb{P}[P(\zeta) \in A] \right| \leq (1 \wedge \zeta^{-1}) \frac{1}{k!} \int w(\mathbf{x} )  \d \mathbf{x}.
\end{align*}
This proves the result for the total variation distance. Moreover, the result for the Wasserstein distance follows in a similar way.
\end{proof}

\subsection{A construction of the scale-free random connection model}\label{ssec:construction}

We now give a construction the scale-free random connection model from a marked Poisson point process. We choose $\mathbb{M}:=[0,1]^{\mathbb{N}\times\mathbb{N}} \times [0,\infty)$ as mark space and $\mathbf{m}$ to be the product distribution of a double sequence of independent random variables uniformly distributed on $[0, 1]$ with an independent random variable $W$ that has the distribution of the weights. Then we consider an independent $\mathbf{m}$-marking $\eta$ of $\mathcal{P}$ (that is a Poisson point process with intensity $\rho \times \mathbf{m}$ on $\mathbb{R}^d\times \mathbb{M}$; see {\it e.g.$\,$}\cite[Theorem 3.5.7]{SW}) and we fix a partition $\{D_i\}_{i\in\mathbb{N}}$ of $\mathbb{R}^d$ that consists of bounded Borel sets. For $x,x'\in\mathbb{R}^d$, we write $x'\lex x$ if $x'$ is smaller than $x$ in the lexicographic order. For $\left(x,\mathbf{u}=(u_{k,l})_{k,l\in\mathbb{N}}, W_x\right)$ and $i\in \mathbb{N}$, note that $\left\{x'\in\mathcal{P}\cap D_i:\, x'\lex x\right\}$ is a.s.$\,$finite since $D_i$ is bounded. {Thus, we can enumerate the elements of this set such that $x_1\lex x_2\lex \dots \lex x_r\lex x$ and set $U(\eta,x,x_j):=u_{i,j}$. Since $\mathcal{P}$ is a.s.$\,$simple, for any pair $\{x,x'\}$ of distinct points in $\mathcal{P}$, we have $x\lex x'$ or $x'\lex x$ thus $U(\eta,x,x')$ (if $x'\lex x$) or $U(\eta,x',x)$ (if $x\lex x'$) is well defined by the above procedure. If $U(\eta,x,y)$ is not defined by this procedure, we set $U(\eta,x,y)=1$.} Then, the random graph $G(\eta)$ with vertex set $\mathcal{P}$ and in which two distinct vertices $x' \lex x$ are connected by an edge if and only if $U(\eta,x,x')\leq 1-\exp\left(-\lambda\frac{W_xW_{x'}}{|x-x'|^\alpha}\right)$ has the law of the scale-free random connection model.

\subsection{The largest unconditional edge-length}\label{ssec:maxedgelength}
The aim in the following is to give a precise description of the magnitude of the unconditional maximum edge-length. For $n\in\N$, we let $B_n = [-n,n]^d$ denote an observation window. Recall that we consider the unconditional length of the longest edge with at least one endpoint in $B_n$, namely
\begin{align*}
 e_n^* &:= \max_{ \{x,y\} \in E_{n}} |x-y|, \qquad E_{n}:= \Big\{ \{x,y\} \in E: x \in B_n \Big\}.
\end{align*}
For given $r\in\rr$ and sequences $(b_n)_{n \in\N}$, $(c_n)_{n \in\N}$, let $r_n = c_nr + b_n$, $n\in\N$. We define the number of exceedances as a random variable $F(n,r)$ given by
\begin{align}\label{def:NbExceed}
\begin{split}
F(n,r)&=  \sum_{x\in \p \cap B_n} \sum_{y\in \p \cap B^C_n} \mathbf{1}_{\{ |x-y| \mathbf{1}_{\{x\leftrightarrow y\}} >r_n \}} + \frac12 \sum_{x\in \p \cap B_n} \sum_{y\in \p \cap B_n} \mathbf{1}_{\{ |x-y| \mathbf{1}_{\{x\leftrightarrow y\}} >r_n \}}\\
& = \sum_{x\in \p \cap B_n} \sum_{y\in \p } \mathbf{1}_{\{ |x-y| \mathbf{1}_{\{x\leftrightarrow y\}} >r_n \}} - \frac12 \sum_{x\in \p \cap B_n} \sum_{y\in \p \cap B_n} \mathbf{1}_{\{ |x-y| \mathbf{1}_{\{x\leftrightarrow y\}} >r_n \}}.
\end{split}
\end{align}
We note that the correction term in \eqref{def:NbExceed} in this representation arises due to the fact that an exceedance with both endpoints in the observation window $B_n$ is counted twice in the number of vertices that are an endpoint of a long edge. Adhering to the construction from Section \ref{ssec:construction} and the notation therein, as in \cite[Section 4.3]{rs1} one can represent the exceedances in the form
\begin{align*}
F(n,r)&=F(\eta)=\sum_{(x,\mathbf{u},W_x), (y,\mathbf{v},W_y)\in \eta}f\left(\{(x,\mathbf{u},W_x), (y,\mathbf{v},W_y)\},\eta\setminus \{(x,\mathbf{u},W_x),(y,\mathbf{v},W_y)\}\right) ,
\end{align*}
where $f$ stands for a suitable indicator function taking into account if a pair of vertices is an exceedance without counting any pair twice, see \cite[Section 4.3]{rs1} for details. Moreover, we set
$$p(x_1,x_2, W_{x_1}, W_{x_2})=\mathbb{E}_{x_1,x_2}\Big[f\left(\{(x_1,\mathbf{u}_1, W_{x_1}), (x_2,\mathbf{u}_2), W_{x_2}\}, \eta\right)\, \big\vert \, W_{x_1}, W_{x_2}\Big] ,$$ 
which is the conditional probability that there is an exceedance at $x_1,x_2$ if we add $x_1,x_2$ equipped with independent random marks to the marked Poisson point process $\eta$. That is by definition we have
\begin{align}\label{eq:estp(x)}
\begin{split}
p(x_1,x_2, W_{x_1}, W_{x_2})&= \mathbf{1}_{\{x_1\in B_n\}} \mathbf{1}_{\{x_2\in B^C_{r_n}(x_1)\}} \left( 1-\exp\left(- \lambda \frac{W_{x_1} W_{x_2}}{|x_1-x_2|^\alpha}\right)\right) .
%\|x_1 - x_2\|^{-\alpha}.
\end{split}
\end{align}

Next, we define the coupled random variables $U_{\mathbf{x},\mathbf{W}}$ and $V_{\mathbf{x},\mathbf{W}}$ after having added two marked points $\mathbf{x}_1=(x_1,\mathbf{u}_1, W_{x_1}), \mathbf{x}_2=(x_2,\mathbf{u}_2, W_{x_2})$. Consider the associated scale-free random connection model $G(\eta \cup \{\mathbf{x}_1, \mathbf{x}_2\} )$. Define the subgraph $G(\eta \cup \{\mathbf{x}_1, \mathbf{x}_2\} )_{|\eta}$ induced by the Poisson points in $\eta$ and observe it has the same distribution as the original random graph. Then $U_{x_1,x_2, W_{x_1}, W_{x_2}}$ is defined as the number of exceedances in the induced graph $G(\eta \cup \{\mathbf{x}_1, \mathbf{x}_2\} )_{|\eta}$:
$$ U_{x_1,x_2, W_{x_1}, W_{x_2}}= \sum_{y\in \mathcal{P} \cap B_n} \sum_{z\in \mathcal{P} } \mathbf{1}_{\{ |z-y| \mathbf{1}_{\{z\leftrightarrow y\}} >r_n \}} - \frac12 \sum_{y\in \mathcal{P} \cap B_n} \sum_{z\in \mathcal{P} \cap B_n} \mathbf{1}_{\{ |z-y| \mathbf{1}_{\{z\leftrightarrow y\}} >r_n \}}.$$
Moreover, $V_{x_1,x_2, W_{x_1}, W_{x_2}}$ is defined as the number of exceedances in $B_n$ in the enlarged graph $G(\eta \cup \{\mathbf{x}_1, \mathbf{x}_2\} )$ other than the one at $x_1,x_2$ (if there is one), namely
\begin{align*}
 V_{x_1,x_2, W_{x_1}, W_{x_2}}&= \sum_{y\in (\mathcal{P} \cup \{x_1, x_2\}) \cap B_n} \sum_{z\in \mathcal{P} \cup \{x_1, x_2\} } \mathbf{1}_{\{ |z-y| \mathbf{1}_{\{z\leftrightarrow y\}} >r_n \}} \\
 & \quad - \frac12 \sum_{y\in (\mathcal{P} \cup \{x_1, x_2\}) \cap B_n} \sum_{z\in \mathcal{P} \cup \{x_1, x_2\} \cap B_n} \mathbf{1}_{\{ |z-y| \mathbf{1}_{\{z\leftrightarrow y\}} >r_n \}}\\
 & \quad - \mathbf{1}_{\{ |x_1-x_2| \mathbf{1}_{\{x_1\leftrightarrow x_2\}} >r_n \}} \Big( \mathbf{1}_{\{x_1 \in B_n\}} + \mathbf{1}_{\{x_2 \in B_n\}}  - \mathbf{1}_{\{x_1 \in B_n\}} \mathbf{1}_{\{x_2 \in B_n\}}\Big) .
\end{align*}

We note that the random variables $U_{x_1,x_2, W_{x_1}, W_{x_2}}$ and $V_{x_1,x_2, W_{x_1}, W_{x_2}}$ fulfill the desired properties. Moreover, it holds $V_{x_1,x_2, W_{x_1}, W_{x_2}}\geq U_{x_1,x_2, W_{x_1}, W_{x_2}}$ by construction. Now we will use Theorem \ref{Th:weightedstein} to prove that the number of exceedances approximates a Poisson distribution. We start by providing an upper bound for $\mathbb{E}[| V_{x_1,x_2, W_{x_1}, W_{x_2}}- U_{x_1,x_2, W_{x_1}, W_{x_2}}|]$, then afterwards an upper bound for $\mathbb{E}[p(x_1,x_2, W_{x_1}, W_{x_2})]$.

We have
\begin{align*}
& \mathbb{E} \left[\vert U_{x_1,x_2, W_{x_1}, W_{x_2}}-V_{x_1,x_2,  W_{x_1}, W_{x_2}}\vert\right]=\mathbb{E}_{x_1,x_2}\left[ V_{x_1,x_2,  W_{x_1}, W_{x_2}}-U_{x_1,x_2,  W_{x_1}, W_{x_2}}\right]\\
&\leq \mathbb{E}_{x_1,x_2}\left[ \sum_{y \in B_n \cap \{x_1,x_2\} } \sum_{z \in  \mathcal{P}}  \mathbf{1}_{\{ |z-y| \mathbf{1}_{\{y\leftrightarrow z\}} > r_n\}} + \sum_{y \in \mathcal{P} \cap B_n } \sum_{z \in  \{x_1,x_2\}}  \mathbf{1}_{\{ |z-y| \mathbf{1}_{\{y\leftrightarrow z\}} > r_n\}}\right] \\
&\leq \mathbf{1}_{\{x_1 \in B_n\}}\int_{\rd} \mathbb{P}_y\left(|x_1-y| \mathbf{1}_{\{y\leftrightarrow x_1\}} > r_n\right) \mathrm{d} y \\
& \quad +  \mathbf{1}_{\{x_2 \in B_n\}}\int_{\rd} \mathbb{P}_y\left(|x_2-y| \mathbf{1}_{\{y\leftrightarrow x_2\}} > r_n\right) \mathrm{d} y \\
&\quad + \int_{B_n} \mathbb{P}_y\left(|x_1-y| \mathbf{1}_{\{y\leftrightarrow x_1\}} > r_n\right) \mathrm{d} y \\
& \quad + \int_{B_n} \mathbb{P}_y\left(|x_2-y| \mathbf{1}_{\{y\leftrightarrow x_2\}} > r_n\right) \mathrm{d} y \\
& \leq C \int_{B_{r_n}^C} \int_{1}^{\infty}\!\!\int_{1}^{\infty}
\Bigl(1-e^{-\lambda wv/|z|}\Bigr)\,v^{-\beta-1}w^{-\beta-1}\,\mathrm{d}w\,\mathrm{d}v \mathrm{d}z,
\end{align*}
for some constant $C$ not depending on $n$. The next lemma gives a suitable bound of the latter integral expression in terms of $\beta$.

\begin{lemma}\label{lem:Jr-upper}
Let $\lambda>0$, $\beta>0$, and define, for $z\in\RR^d$,
\[
J(z)\,:=\,
\int_{1}^{\infty}\!\!\int_{1}^{\infty}
\Bigl(1-e^{-\lambda wv/|z|^\alpha}\Bigr)\,v^{-\beta-1}w^{-\beta-1}\,\mathrm{d}w\,\mathrm{d}v.
\]
Then there exists a constant $C=C(\lambda,\beta)\in(0,\infty)$ such that, for all $|z|$ sufficiently large,
\[
J(z)\ \le\
C\begin{cases}
|z|^{-\alpha\beta}\log |z|^\alpha, & \beta<1,\\[1mm]
|z|^{-\alpha}\log^2 |z|^\alpha, & \beta=1,\\[1mm]
|z|^{-\alpha}, & \beta>1.
\end{cases}
\]
\end{lemma}

\begin{proof}
We apply integration by parts in the inner $w$--integral. For fixed $v\ge 1$,
\[
\int_{1}^{\infty} \bigl(1-e^{-\lambda v w/|z|^\alpha}\bigr)\,w^{-\beta-1}\,\mathrm{d}w
=\frac{1}{\beta}\bigl(1-e^{-\lambda v/|z|^\alpha}\bigr)
+\frac{1}{\beta}\frac{\lambda v}{|z|^\alpha}\int_{1}^{\infty} w^{-\beta}e^{-\lambda v w/|z|^\alpha}\,\mathrm{d}w.
\]
Multiplying by $v^{-\beta-1}$ and integrating over $v\in[1,\infty)$ yields the exact decomposition
\[
J(z)=J^{(1)}(z)+J^{(2)}(z),
\]
where
\[
J^{(1)}(z):=\frac{1}{\beta}\int_{1}^{\infty} \bigl(1-e^{-\lambda v/|z|^\alpha}\bigr)\,v^{-\beta-1}\,\mathrm{d}v,
\qquad
J^{(2)}(z):=\frac{1}{\beta}\frac{\lambda}{|z|^\alpha}\int_{1}^{\infty}\!\!\int_{1}^{\infty}
v^{-\beta}w^{-\beta}e^{-\lambda v w/|z|^\alpha}\,\mathrm{d}w\,\mathrm{d}v.
\]

\smallskip\noindent
\emph{Step 1: bound on $J^{(1)}(z)$.}
Applying integration by parts to the $v$--integral gives
\[
\int_{1}^{\infty} \bigl(1-e^{-\lambda v/|z|^\alpha}\bigr)\,v^{-\beta-1}\,\mathrm{d}v
=\frac{1}{\beta}\bigl(1-e^{-\lambda/|z|^\alpha}\bigr)
+\frac{1}{\beta}\frac{\lambda}{|z|^\alpha}\int_{1}^{\infty} v^{-\beta}e^{-\lambda v/|z|^\alpha}\,\mathrm{d}v.
\]
Hence
\[
J^{(1)}(z)
=\frac{1}{\beta^2}\bigl(1-e^{-\lambda/|z|^\alpha}\bigr)
+\frac{1}{\beta^2}\frac{\lambda}{|z|^\alpha}\int_{1}^{\infty} v^{-\beta}e^{-\lambda v/|z|^\alpha}\,\mathrm{d}v.
\]
With the substitution $u=\lambda v/|z|^\alpha$,
\[
\int_{1}^{\infty} v^{-\beta}e^{-\lambda v/|z|^\alpha}\,\mathrm{d}v
=\Bigl(\frac{|z|^\alpha}{\lambda}\Bigr)^{1-\beta}\int_{\lambda/|z|^\alpha}^{\infty} u^{-\beta}e^{-u}\,\mathrm{d}u.
\]
For $|z|$ large enough (so that $\lambda/|z|^\alpha\le 1$), we split the last integral at $1$ and use
$e^{-u}\le 1$ on $(0,1]$ and $\int_{1}^{\infty}u^{-\beta}e^{-u}\,\mathrm{d}u<\infty$ to obtain
\[
\int_{\lambda/|z|^\alpha}^{\infty} u^{-\beta}e^{-u}\,\mathrm{d}u
\le
C\begin{cases}
1, & \beta<1,\\[1mm]
1+\log |z|^\alpha, & \beta=1,\\[1mm]
\bigl(\lambda/|z|^\alpha\bigr)^{1-\beta}, & \beta>1.
\end{cases}
\]
Inserting this bound and using $1-e^{-\lambda/|z|^\alpha}\le \lambda/|z|^\alpha$ yields
\[
J^{(1)}(z)\ \le\ C
\begin{cases}
|z|^{-\alpha\beta}, & \beta<1,\\[1mm]
|z|^{-\alpha}\log |z|^\alpha, & \beta=1,\\[1mm]
|z|^{-\alpha}, & \beta>1.
\end{cases}
\]

\smallskip\noindent
\emph{Step 2: bound on $J^{(2)}(z)$.}
Integrating out $w$ and using the substitution $u=\lambda v w/|z|^\alpha$ yields, for each $v\ge 1$,
\[
\int_{1}^{\infty} w^{-\beta}e^{-\lambda v w/|z|^\alpha}\,\mathrm{d}w
=\Bigl(\frac{\lambda v}{|z|^\alpha}\Bigr)^{\beta-1}
\int_{\lambda v/|z|^\alpha}^{\infty} u^{-\beta}e^{-u}\,\mathrm{d}u.
\]
Hence
\[
J^{(2)}(z)
=\frac{1}{\beta}\Bigl(\frac{\lambda}{|z|^\alpha}\Bigr)^{\beta}
\int_{1}^{\infty} \frac{1}{v}\left(\int_{\lambda v/|z|^\alpha}^{\infty} u^{-\beta}e^{-u}\,\mathrm{d}u\right)\,\mathrm{d}v.
\]
Split the $v$--integral at $v=|z|^\alpha/\lambda$. For $0<x\le 1$ we use the elementary bound
\[
\int_{x}^{\infty} u^{-\beta}e^{-u}\,\mathrm{d}u
\le C
\begin{cases}
1, & \beta<1,\\[1mm]
1+\log(1/x), & \beta=1,\\[1mm]
x^{1-\beta}, & \beta>1,
\end{cases}
\]
while for $x\ge 1$ we use $\int_{x}^{\infty} u^{-\beta}e^{-u}\,\mathrm{d}u\le C e^{-x}$.
This gives
\[
J^{(2)}(z)\le C\Bigl(\frac{\lambda}{|z|^\alpha}\Bigr)^{\beta}
\left(
\int_{1}^{|z|^\alpha/\lambda}\frac{h_\beta(\lambda v/|z|^\alpha)}{v}\,\mathrm{d}v
+\int_{|z|^\alpha/\lambda}^{\infty}\frac{e^{-\lambda v/|z|^\alpha}}{v}\,\mathrm{d}v
\right),
\]
where $h_\beta(x)\equiv 1$ if $\beta<1$, $h_\beta(x)=1+\log(1/x)$ if $\beta=1$, and
$h_\beta(x)=x^{1-\beta}$ if $\beta>1$.
The second integral is bounded uniformly in $|z|$, and the first one yields
\[
\int_{1}^{|z|^\alpha/\lambda}\frac{h_\beta(\lambda v/|z|^\alpha)}{v}\,\mathrm{d}v
\le
C\begin{cases}
\log |z|^\alpha, & \beta<1,\\[1mm]
\log^2 |z|^\alpha, & \beta=1,\\[1mm]
(|z|^\alpha/\lambda)^{1-\beta}, & \beta>1.
\end{cases}
\]
Consequently, for $|z|$ sufficiently large,
\[
J^{(2)}(z)\ \le\ C
\begin{cases}
|z|^{-\alpha\beta}\log |z|^\alpha, & \beta<1,\\[1mm]
|z|^{-\alpha}\log^2 |z|^\alpha, & \beta=1,\\[1mm]
|z|^{-\alpha}, & \beta>1.
\end{cases}
\]
\end{proof}

Next we consider $\mathbb{E}[p(x_1,x_2, W_{x_1}, W_{x_2})]$. By \eqref{eq:estp(x)} we have
\begin{align*}
    \mathbb{E}[p(x_1,x_2, W_{x_1}, W_{x_2})] &= \mathbf{1}_{\{x_1\in B_n\}} \mathbf{1}_{\{x_2\in B^C_{r_n}(x_1)\}} \mathbb{E}\left[ \left( 1-\exp\left(- \lambda \frac{W_{x_1} W_{x_2}}{|x_1-x_2|^\alpha}\right)\right)\right] .%\\
   % & \leq \lambda \mathbb{E}[W_0]^2  \mathbf{1}_{\{x_1\in B_n\}} \mathbf{1}_{\{x_2\in B^C_{r_n}(x_1)\}} |x_1-x_2|^{-\alpha}.
\end{align*}

In the same way as Lemma \ref{lem:Jr-upper} we obtain the following.

\begin{lemma}\label{lem:Jr-upper2}
There exists a constant $C=C(\lambda,\beta)\in(0,\infty)$ such that, for sufficiently large $n$,
\[
 \mathbb{E}[p(x_1,x_2, W_{x_1}, W_{x_2})] \le\
C \mathbf{1}_{\{x_1\in B_n\}} \mathbf{1}_{\{x_2\in B^C_{r_n}(x_1)\}} \begin{cases}
|x_1-x_2|^{-\alpha\beta}\log |x_1-x_2|^\alpha, & \beta<1,\\[1mm]
|x_1-x_2|^{-\alpha}\log^2 |x_1-x_2|^\alpha, & \beta=1,\\[1mm]
|x_1-x_2|^{-\alpha}, & \beta>1.
\end{cases}
\]
\end{lemma}

Using Lemma \ref{lem:Jr-upper2} we now prove the next statement.

\begin{lemma}\label{lem:double-int-bound}
There exists $C\in(0,\infty)$, depending only on $(d,\alpha,\beta,\lambda)$, such that
\[
\int_{\RR^d}\int_{\RR^d}\EE\!\left[p(x_1,x_2,W_{x_1},W_{x_2})\right]\,
\mathrm{d}x_2\,\mathrm{d}x_1
\ \le\
C\,(2n)^d\,
\begin{cases}
r_n^{\,d-\alpha\beta}\log(r_n^\alpha), & \beta<1,\\[1mm]
r_n^{\,d-\alpha}\log^2(r_n^\alpha), & \beta=1,\\[1mm]
r_n^{\,d-\alpha}, & \beta>1.
\end{cases}
\]
\end{lemma}

\begin{proof}
Define
\[
g_\beta(r):=
\begin{cases}
r^{-\alpha\beta}\log(r^\alpha), & \beta<1,\\[1mm]
r^{-\alpha}\log^2(r^\alpha), & \beta=1,\\[1mm]
r^{-\alpha}, & \beta>1.
\end{cases}
\]
By Lemma \ref{lem:Jr-upper2},
\[
\int_{\RR^d}\int_{\RR^d}\EE[p(x_1,x_2,W_{x_1},W_{x_2})]\,\mathrm{d}x_2\,\mathrm{d}x_1
\le
C\int_{B_n}\left(\int_{B_{r_n}^c(x_1)} g_\beta(|x_1-x_2|)\,\mathrm{d}x_2\right)\mathrm{d}x_1.
\]
For fixed $x_1\in B_n$ we substitute $z=x_2-x_1$ and obtain
\[
\int_{B_{r_n}^c(x_1)} g_\beta(|x_1-x_2|)\,\mathrm{d}x_2
=
\int_{|z|>r_n} g_\beta(|z|)\,\mathrm{d}z,
\]
which is independent of $x_1$. Hence
\begin{align}
\label{eq:reduce-tail-int}
\int_{\RR^d}\int_{\RR^d}\EE[p(x_1,x_2,W_{x_1},W_{x_2})]\,\mathrm{d}x_2\,\mathrm{d}x_1
&\le
C\,|B_n|\,\int_{|z|>r_n} g_\beta(|z|)\,\mathrm{d}z\\
&=
C\,(2n)^d \int_{|z|>r_n} g_\beta(|z|)\,\mathrm{d}z.\nonumber
\end{align}
We estimate the remaining tail integral in polar coordinates:
\[
\int_{|z|>r_n} g_\beta(|z|)\,\mathrm{d}z
=
d \omega_d\int_{r_n}^{\infty} r^{d-1} g_\beta(r)\,\mathrm{d}r,
\]
where $\omega_d$ denotes the $d$--dimensional volume of the Euclidean unit ball.

\smallskip\noindent
\emph{Case $\beta>1$.}
Here $g_\beta(r)=r^{-\alpha}$ and $\alpha>d$ implies
\[
\int_{r_n}^{\infty} r^{d-1-\alpha}\,\mathrm{d}r
=\frac{1}{\alpha-d}\,r_n^{\,d-\alpha}.
\]

\smallskip\noindent
\emph{Case $\beta=1$.}
Since $\log(r^\alpha)=\alpha\log r$, we obtain, with $p:=\alpha-d>0$,
\[
\int_{r_n}^{\infty} r^{d-1-\alpha}\log^2(r^\alpha)\,\mathrm{d}r
=\alpha^2\int_{r_n}^{\infty} r^{-p-1}(\log r)^2\,\mathrm{d}r.
\]
An integration-by-parts recursion yields the exact identity
\[
\int_{r_n}^{\infty} r^{-p-1}(\log r)^2\,\mathrm{d}r
=\frac{1}{p}\,r_n^{-p}\Bigl((\log r_n)^2+\frac{2}{p}\log r_n+\frac{2}{p^2}\Bigr)
\le C\,r_n^{-p}(\log r_n)^2,
\]
and therefore
\[
\int_{r_n}^{\infty} r^{d-1-\alpha}\log^2(r^\alpha)\,\mathrm{d}r
\le C\,r_n^{\,d-\alpha}\log^2(r_n^\alpha).
\]

\smallskip\noindent
\emph{Case $\beta<1$.}
Here $g_\beta(r)=r^{-\alpha\beta}\log(r^\alpha)$ and $\alpha\beta>d$ implies
\[
\int_{r_n}^{\infty} r^{d-1-\alpha\beta}\log(r^\alpha)\,\mathrm{d}r
=\alpha\int_{r_n}^{\infty} r^{d-1-\alpha\beta}\log r\,\mathrm{d}r
\le C\,r_n^{\,d-\alpha\beta}\log(r_n^\alpha),
\]
where the last estimate follows by a single integration by parts, using that
$d-\alpha\beta<0$.

\smallskip
Combining these bounds with \eqref{eq:reduce-tail-int} proves the claim.
\end{proof}

Finally, combining Lemma \ref{lem:Jr-upper}, Lemma \ref{lem:double-int-bound} and using Theorem \ref{Th:weightedstein}
we have thus obtained the following upper bounds for the total variation and Wasserstein distance between the number of exceedances and a Poisson random variable with equal mean.

\begin{prop}\label{prop:Poisson-approx-Fnr}
Let $P(n,r)$ be a Poisson random variable with mean $\zeta_{n,r}:=\EE\big[F(n,r)\big]$. Define
\[
\iota_n:=
\begin{cases}
r_n^{\,d-\alpha\beta}\log(r_n^\alpha), & \beta<1,\\[1mm]
r_n^{\,d-\alpha}\log^2(r_n^\alpha), & \beta=1,\\[1mm]
r_n^{\,d-\alpha}, & \beta>1.
\end{cases}
\]
Then,
\[
d_{\mathrm{TV}}\big(F(n,r),P(n,r)\big)
\ \le\
C\,\min\{1,\zeta_{n,r}^{-1}\}\, n^d \iota_n^{\,2},
\]
and
\[
d_{\mathrm{W}}\big(F(n,r),P(n,r)\big)
\ \le\
C\,\min\{1,\zeta_{n,r}^{-1/2}\}\,n^d \iota_n^{\,2},
\]
for some constant $C\in (0, \infty)$ independent of $n$.
\end{prop}

\subsection{The expected number of exceedances}\label{ssec:expectedexceedances}
We now calculate
\begin{align*}
\mathbb{E}[F(n,r)]& = \mathbb{E}\left [\sum_{x\in V \cap B_n} \sum_{y\in V } \mathbf{1}_{\{ \|x-y\| \mathbf{1}_{\{x\leftrightarrow y\}} >r_n \}} - \frac12 \sum_{x\in V \cap B_n} \sum_{y\in V \cap B_n} \mathbf{1}_{\{ \|x-y\| \mathbf{1}_{\{x\leftrightarrow y\}} >r_n \}} \right] \\
& =: I_1 - I_2,
\end{align*}
with 
\begin{align*}
I_1& = \mathbb{E}\left [\sum_{x\in V \cap B_n} \sum_{y\in V } \mathbf{1}_{\{ \|x-y\| \mathbf{1}_{\{x\leftrightarrow y\}} >r_n \}}  \right] 
\end{align*}
and correction term
\begin{align*}
I_2& =  \mathbb{E}\left [\frac12 \sum_{x\in V \cap B_n} \sum_{y\in V \cap B_n} \mathbf{1}_{\{ \|x-y\| \mathbf{1}_{\{x\leftrightarrow y\}} >r_n \}} \right].
\end{align*}
We have
\begin{align*}
I_1& = \int_{B_n} \int_{B_{r_n}^C (x)} \mathbb{P} (x\leftrightarrow y) \d y \d x \\
& = \int_{B_n} \int_{B_{r_n}^C (x)} \int_1^\infty\int_1^\infty\mathbb{P} (x\leftrightarrow y \vert W_x = w, W_y = v) \d F_{W_x}(w) \d F_{W_y}(v) \d y \d x \\
& = \int_1^\infty\int_1^\infty \int_{B_n} \int_{B_{r_n}^C (x)} \left( 1-\exp\left(-\lambda \frac{wv}{|x-y|^\alpha} \right) \right) \d y \d x \d F_{W}(w) \d F_{W}(v) \\
& = \int_1^\infty\int_1^\infty (2n)^d \int_{B_{r_n}^C (0)} \left( 1-\exp\left(-\lambda \frac{wv}{|y|^\alpha} \right) \right) \d y  \d F_{W}(w) \d F_{W}(v) .
\end{align*}

Using polar coordinates and integration by parts yields the exact decomposition
\begin{equation*}\label{eq:I1-decomp}
I_1 = -I_1^{(A)} + I_1^{(B)},
\end{equation*}
where
\begin{equation}\label{eq:I1A-def}
I_1^{(A)}
:=
(2n)^d\,\omega_d\,r_n^{d}
\int_{1}^{\infty}\!\!\int_{1}^{\infty}
\Bigl(1-e^{-\lambda wv/r_n^\alpha}\Bigr)\,\mathrm{d}F_W(w)\,\mathrm{d}F_W(v),
\end{equation}
and
\begin{equation}\label{eq:I1B-def}
I_1^{(B)}
:=
(2n)^d\,\omega_d\,r_n^{d}\,\frac{\lambda}{r_n^\alpha}
\int_{1}^{\infty}\!\!\int_{1}^{\infty}
wv\left(\int_{0}^{1} t^{-d/\alpha}e^{-\lambda wv t/r_n^\alpha}\,\mathrm{d}t\right)\,
\mathrm{d}F_W(w)\,\mathrm{d}F_W(v).
\end{equation}

\medskip

\begin{lemma}\label{lem:I1A-asymptotics}
Define $I_1^{(A)}$ by \eqref{eq:I1A-def}. Then, as $n\to\infty$,
\[
I_1^{(A)}\sim (2n)^d\,\omega_d\,
\begin{cases}
\beta\,\lambda^\beta\,\Gamma(1-\beta)\, r_n^{\,d-\alpha\beta}\,\log(r_n^\alpha),
& 0<\beta<1,\\[1mm]
\dfrac12\,\beta\,\lambda\, r_n^{\,d-\alpha}\,\log^2(r_n^\alpha),
& \beta=1,\\[2mm]
\lambda \dfrac{\beta^2}{(\beta-1)^2}\, r_n^{\,d-\alpha},
& \beta>1,
\end{cases}
\]
\end{lemma}

\begin{proof}
Write \eqref{eq:I1A-def} in terms of the Pareto density:
\[
I_1^{(A)}
=(2n)^d\,\omega_d\,r_n^{d}\,\beta^2
\int_{1}^{\infty}\!\!\int_{1}^{\infty}
\Bigl(1-e^{-\lambda wv/r_n^\alpha}\Bigr)\,w^{-\beta-1}v^{-\beta-1}\,
\mathrm{d}w\,\mathrm{d}v.
\]
Fix $v\ge 1$ and set $a:=\lambda v/r_n^\alpha$. Integration by parts in the $w$--integral gives
\[
\int_{1}^{\infty}\bigl(1-e^{-aw}\bigr)\,w^{-\beta-1}\,\mathrm{d}w
=\frac{1}{\beta}\bigl(1-e^{-a}\bigr)
+\frac{a}{\beta}\int_{1}^{\infty} w^{-\beta}e^{-aw}\,\mathrm{d}w.
\]
Multiplying by $\beta^2 v^{-\beta-1}$ and integrating over $v\in[1,\infty)$ yields
\begin{equation*}\label{eq:I1A-split}
I_1^{(A)}=I_{1,1}^{(A)}+I_{1,2}^{(A)},
\end{equation*}
where
\[
I_{1,1}^{(A)}
:=(2n)^d\,\omega_d\,r_n^{d}\,\beta
\int_{1}^{\infty}\bigl(1-e^{-\lambda v/r_n^\alpha}\bigr)\,v^{-\beta-1}\,\mathrm{d}v,
\]
and
\[
I_{1,2}^{(A)}
:=(2n)^d\,\omega_d\,r_n^{d}\,\beta
\int_{1}^{\infty}\frac{\lambda v}{r_n^\alpha}
\left(\int_{1}^{\infty} w^{-\beta}e^{-\lambda v w/r_n^\alpha}\,\mathrm{d}w\right)
v^{-\beta-1}\,\mathrm{d}v.
\]

\smallskip\noindent
\emph{Step 1: asymptotics of $I_{1,1}^{(A)}$.}
Set $a:=\lambda/r_n^\alpha\downarrow 0$. A second integration by parts in $v$ yields
\[
\int_{1}^{\infty}\bigl(1-e^{-av}\bigr)v^{-\beta-1}\,\mathrm{d}v
=\frac{1}{\beta}\bigl(1-e^{-a}\bigr)+\frac{a}{\beta}\int_{1}^{\infty} v^{-\beta}e^{-av}\,\mathrm{d}v.
\]
Substituting $u=av$ in the last integral gives
\[
a\int_{1}^{\infty} v^{-\beta}e^{-av}\,\mathrm{d}v
=a^\beta\int_{a}^{\infty} u^{-\beta}e^{-u}\,\mathrm{d}u.
\]
Consequently,
\[
I_{1,1}^{(A)}
=(2n)^d\,\omega_d\,r_n^{d}\Bigl[\bigl(1-e^{-a}\bigr)+a^\beta\int_{a}^{\infty} u^{-\beta}e^{-u}\,\mathrm{d}u\Bigr].
\]
As $a\downarrow 0$, we have $1-e^{-a}\sim a$. Moreover,
\[
\int_{a}^{\infty} u^{-\beta}e^{-u}\,\mathrm{d}u
\sim
\begin{cases}
\Gamma(1-\beta), & 0<\beta<1,\\[1mm]
\log(1/a), & \beta=1,\\[1mm]
\dfrac{1}{\beta-1}\,a^{1-\beta}, & \beta>1,
\end{cases}
\]
which implies
\[
\bigl(1-e^{-a}\bigr)+a^\beta\int_{a}^{\infty} u^{-\beta}e^{-u}\,\mathrm{d}u
\sim
\begin{cases}
\Gamma(1-\beta)\,a^\beta, & 0<\beta<1,\\[1mm]
a\log(1/a), & \beta=1,\\[1mm]
a\Bigl(1+\dfrac{1}{\beta-1}\Bigr)=a\,\dfrac{\beta}{\beta-1}, & \beta>1.
\end{cases}
\]
Therefore,
\[
I_{1,1}^{(A)}\sim (2n)^d\,\omega_d
\begin{cases}
\Gamma(1-\beta)\,\lambda^\beta\,r_n^{d-\alpha\beta}, & 0<\beta<1,\\[1mm]
\lambda\,r_n^{d-\alpha}\log(r_n^\alpha), & \beta=1,\\[1mm]
\lambda\,\dfrac{\beta}{\beta-1}\,r_n^{d-\alpha}, & \beta>1.
\end{cases}
\]

\smallskip\noindent
\emph{Step 2: asymptotics of $I_{1,2}^{(A)}$.}
Set $a:=\lambda/r_n^\alpha\downarrow 0$. From the definition,
\begin{align*}
I_{1,2}^{(A)}
&=(2n)^d\,\omega_d\,r_n^{d}\,\beta
\int_{1}^{\infty}\frac{\lambda v}{r_n^\alpha}
\left(\int_{1}^{\infty} w^{-\beta}e^{-\lambda v w/r_n^\alpha}\,\mathrm{d}w\right)
v^{-\beta-1}\,\mathrm{d}v \\
&=(2n)^d\,\omega_d\,r_n^{d}\,\beta a
\int_{1}^{\infty}\!\!\int_{1}^{\infty} v^{-\beta}w^{-\beta}e^{-a v w}\,\mathrm{d}w\,\mathrm{d}v.
\end{align*}
By Tonelli's theorem and the substitution $u=vw$ in the inner $w$--integral (for fixed $v$),
\[
\int_{1}^{\infty} w^{-\beta}e^{-a v w}\,\mathrm{d}w
=\int_{v}^{\infty}\Bigl(\frac{u}{v}\Bigr)^{-\beta}e^{-a u}\frac{\mathrm{d}u}{v}
=v^{\beta-1}\int_{v}^{\infty} u^{-\beta}e^{-a u}\,\mathrm{d}u.
\]
Hence
\[
\int_{1}^{\infty}\!\!\int_{1}^{\infty} v^{-\beta}w^{-\beta}e^{-a v w}\,\mathrm{d}w\,\mathrm{d}v
=\int_{1}^{\infty}\frac{1}{v}\left(\int_{v}^{\infty} u^{-\beta}e^{-a u}\,\mathrm{d}u\right)\mathrm{d}v.
\]
Swapping the order of integration over the region $\{(v,u):1\le v\le u<\infty\}$ yields the exact identity
\begin{equation*}\label{eq:Ka-log}
\int_{1}^{\infty}\!\!\int_{1}^{\infty} v^{-\beta}w^{-\beta}e^{-a v w}\,\mathrm{d}w\,\mathrm{d}v
=\int_{1}^{\infty} u^{-\beta}e^{-a u}\left(\int_{1}^{u}\frac{\mathrm{d}v}{v}\right)\mathrm{d}u
=\int_{1}^{\infty} u^{-\beta}e^{-a u}\log u\,\mathrm{d}u.
\end{equation*}
Consequently,
\begin{equation}\label{eq:I12A-Ka}
I_{1,2}^{(A)}
=(2n)^d\,\omega_d\,r_n^{d}\,\beta a
\int_{1}^{\infty} u^{-\beta}e^{-a u}\log u\,\mathrm{d}u.
\end{equation}

\smallskip
\noindent
We now evaluate the last integral as $a\downarrow 0$.

\smallskip\noindent
{First we consider the case $0<\beta<1$.}
With the substitution $t=a u$,
\[
\int_{1}^{\infty} u^{-\beta}e^{-a u}\log u\,\mathrm{d}u
=a^{\beta-1}\int_{a}^{\infty} t^{-\beta}e^{-t}\bigl(\log t-\log a\bigr)\,\mathrm{d}t.
\]
Since $\int_{0}^{\infty}t^{-\beta}e^{-t}\,\mathrm{d}t=\Gamma(1-\beta)$ and
$\int_{0}^{\infty}t^{-\beta}e^{-t}|\log t|\,\mathrm{d}t<\infty$ for $\beta\in(0,1)$,
we obtain
\[
\int_{1}^{\infty} u^{-\beta}e^{-a u}\log u\,\mathrm{d}u
\sim a^{\beta-1}\Gamma(1-\beta)\log(1/a).
\]
Inserting into \eqref{eq:I12A-Ka} gives
\[
I_{1,2}^{(A)}
\sim (2n)^d\,\omega_d\,\beta\,\lambda^\beta\,\Gamma(1-\beta)\,
r_n^{d-\alpha\beta}\,\log(r_n^\alpha).
\]

\smallskip\noindent
{Now assume $\beta=1$.}
Again with $t=a u$,
\[
\int_{1}^{\infty} u^{-1}e^{-a u}\log u\,\mathrm{d}u
=\int_{a}^{\infty} e^{-t}\frac{\log(t/a)}{t}\,\mathrm{d}t.
\]
Splitting at $1$ and using $e^{-t}\le 1$, we have
\[
\int_{a}^{\infty} e^{-t}\frac{\log(t/a)}{t}\,\mathrm{d}t
=\int_{a}^{1}\frac{\log(t/a)}{t}\,\mathrm{d}t + O(1).
\]
The substitution $t=a e^{x}$ yields the exact identity
\[
\int_{a}^{1}\frac{\log(t/a)}{t}\,\mathrm{d}t
=\int_{0}^{\log(1/a)} x\,\mathrm{d}x
=\frac12\log^2(1/a),
\]
hence
\[
\int_{1}^{\infty} u^{-1}e^{-a u}\log u\,\mathrm{d}u
\sim \frac12\log^2(1/a).
\]
Inserting into \eqref{eq:I12A-Ka} gives
\[
I_{1,2}^{(A)}
\sim \frac12(2n)^d\,\omega_d\,\beta\,\lambda\,r_n^{d-\alpha}\,\log^2(r_n^\alpha).
\]

\smallskip\noindent
{It remains to consider $\beta>1$.}
Since $u^{-\beta}\log u$ is integrable on $[1,\infty)$, dominated convergence yields
\[
\int_{1}^{\infty} u^{-\beta}e^{-a u}\log u\,\mathrm{d}u
\longrightarrow \int_{1}^{\infty} u^{-\beta}\log u\,\mathrm{d}u
=\frac{1}{(\beta-1)^2},
\qquad a\downarrow 0.
\]
Inserting into \eqref{eq:I12A-Ka} gives
\[
I_{1,2}^{(A)}
\sim (2n)^d\,\omega_d\,\beta\,\lambda\,\frac{1}{(\beta-1)^2}\,r_n^{d-\alpha}.
\]
In summary, we have
\[
I_{1,2}^{(A)}
\sim (2n)^d\,\omega_d\,
\begin{cases}
\beta\,\lambda^\beta\,\Gamma(1-\beta)\, r_n^{\,d-\alpha\beta}\,\log(r_n^\alpha),
& 0<\beta<1,\\[1mm]
\dfrac12\,\beta\,\lambda\, r_n^{\,d-\alpha}\,\log^2(r_n^\alpha),
& \beta=1,\\[2mm]
\beta\,\lambda\,(\beta-1)^{-2}\, r_n^{\,d-\alpha},
& \beta>1,
\end{cases}
\qquad n\to\infty.
\]
Since
\[
I_1^{(A)}=I_{1,1}^{(A)}+I_{1,2}^{(A)},
\]
and therefore, as $n\to\infty$,
\[
I_1^{(A)}\sim (2n)^d\,\omega_d\,
\begin{cases}
\beta\,\lambda^\beta\,\Gamma(1-\beta)\, r_n^{\,d-\alpha\beta}\,\log(r_n^\alpha),
& 0<\beta<1,\\[1mm]
\dfrac12\,\beta\,\lambda\, r_n^{\,d-\alpha}\,\log^2(r_n^\alpha),
& \beta=1,\\[2mm]
\lambda\left(\dfrac{\beta}{\beta-1}+\dfrac{\beta}{(\beta-1)^2}\right)\, r_n^{\,d-\alpha},
& \beta>1,
\end{cases}
\qquad n\to\infty.
\]

\end{proof}

\begin{lemma}\label{lem:I1B-asymptotics}
Define $I_1^{(B)}$ as in \eqref{eq:I1B-def}. Then, as $n\to\infty$,
\[
I_1^{(B)}\sim (2n)^d\,\omega_d\,
\begin{cases}
\beta^2\,\Gamma(1-\beta)\,\lambda^\beta\,\dfrac{\alpha}{\alpha\beta-d}\,
r_n^{\,d-\alpha\beta}\,\log(r_n^\alpha),
& 0<\beta<1,\\[2mm]
\dfrac12\,\lambda\,\dfrac{\alpha}{\alpha-d}\,r_n^{\,d-\alpha}\,\log^2(r_n^\alpha),
& \beta=1,\\[2mm]
\beta^2\,\lambda\,\dfrac{\alpha}{\alpha-d}\,(\beta-1)^{-2}\,r_n^{\,d-\alpha},
& \beta>1.
\end{cases}
\]
\end{lemma}

\begin{proof}
Write $a=\lambda/r_n^\alpha$. Inserting the Pareto density and using Tonelli's theorem,
\begin{align*}
I_1^{(B)}
&=(2n)^d\,\omega_d\,r_n^{d}\,a\,\beta^2
\int_{1}^{\infty}\!\!\int_{1}^{\infty}
w^{-\beta}v^{-\beta}
\left(\int_{0}^{1} t^{-d/\alpha}e^{-a wv\, t}\,\mathrm{d}t\right)\mathrm{d}w\,\mathrm{d}v \\
&=(2n)^d\,\omega_d\,r_n^{d}\,a\,\beta^2
\int_{0}^{1} t^{-d/\alpha}
\left(\int_{1}^{\infty}\!\!\int_{1}^{\infty}
w^{-\beta}v^{-\beta}e^{-a t\, wv}\,\mathrm{d}w\,\mathrm{d}v\right)\mathrm{d}t.
\end{align*}
Introduce
\[
K(x)=\int_{1}^{\infty}\!\!\int_{1}^{\infty} w^{-\beta}v^{-\beta}e^{-x wv}\,\mathrm{d}w\,\mathrm{d}v,
\qquad x>0,
\]
so that
\begin{equation}\label{eq:I1B-K}
I_1^{(B)}=(2n)^d\,\omega_d\,r_n^{d}\,a\,\beta^2\int_{0}^{1} t^{-d/\alpha}K(a t)\,\mathrm{d}t.
\end{equation}
As in the proof of Lemma~\ref{lem:I1A-asymptotics} (Step~2), a Tonelli argument and the substitution $u=vw$
show that, for $x>0$,
\begin{equation}\label{eq:K-log-form}
K(x)=\int_{1}^{\infty} u^{-\beta}e^{-x u}\log u\,\mathrm{d}u.
\end{equation}
We now evaluate \eqref{eq:I1B-K} using the small--$x$ asymptotics of \eqref{eq:K-log-form}.

\smallskip\noindent
First {consider $0<\beta<1$.}
From Lemma~\ref{lem:I1A-asymptotics} (Step~2), as $x\downarrow 0$,
\[
K(x)\sim \Gamma(1-\beta)\,x^{\beta-1}\log(1/x).
\]
Hence, uniformly for $t\in(0,1]$,
\[
K(at)\sim \Gamma(1-\beta)\,a^{\beta-1}t^{\beta-1}\bigl(\log(1/a)+\log(1/t)\bigr).
\]
Since $\alpha\beta>d$, we have $\beta-d/\alpha>0$ and therefore
$t^{\beta-1-d/\alpha}$ is integrable on $(0,1)$. Consequently,
\begin{align*}
\int_{0}^{1} t^{-d/\alpha}K(at)\,\mathrm{d}t
&\sim \Gamma(1-\beta)\,a^{\beta-1}
\int_{0}^{1} t^{\beta-1-d/\alpha}\bigl(\log(1/a)+\log(1/t)\bigr)\,\mathrm{d}t \\
&=\Gamma(1-\beta)\,a^{\beta-1}
\left(\frac{\log(1/a)}{\beta-d/\alpha}+\frac{1}{(\beta-d/\alpha)^2}\right),
\end{align*}
where we used $\int_0^1 t^{p-1}\,\mathrm{d}t=p^{-1}$ and
$\int_0^1 t^{p-1}\log(1/t)\,\mathrm{d}t=p^{-2}$ with $p=\beta-d/\alpha$.
Inserting into \eqref{eq:I1B-K} and using $r_n^d a^\beta=\lambda^\beta r_n^{d-\alpha\beta}$ yields
\[
I_1^{(B)}\sim (2n)^d\,\omega_d\,\beta^2\Gamma(1-\beta)\lambda^\beta\,
\frac{1}{\beta-d/\alpha}\,r_n^{d-\alpha\beta}\log(r_n^\alpha).
\]
Finally, $1/(\beta-d/\alpha)=\alpha/(\alpha\beta-d)$, which gives the claimed constant.

\smallskip\noindent
Now{ let $\beta=1$.}
From Lemma~\ref{lem:I1A-asymptotics} (Step~2), as $x\downarrow 0$,
\[
K(x)\sim \frac12\log^2(1/x).
\]
Thus
\[
\int_0^1 t^{-d/\alpha}K(at)\,\mathrm{d}t
\sim \frac12\int_0^1 t^{-d/\alpha}\bigl(\log(1/a)+\log(1/t)\bigr)^2\,\mathrm{d}t.
\]
The leading term is obtained by keeping $\log^2(1/a)$ and using
$\int_0^1 t^{-d/\alpha}\,\mathrm{d}t = (1-d/\alpha)^{-1}=\alpha/(\alpha-d)$, hence
\[
\int_0^1 t^{-d/\alpha}K(at)\,\mathrm{d}t
\sim \frac12\,\frac{\alpha}{\alpha-d}\,\log^2(1/a).
\]
Inserting into \eqref{eq:I1B-K} and using $r_n^d a=\lambda r_n^{d-\alpha}$ yields
\[
I_1^{(B)}\sim (2n)^d\,\omega_d\,\beta^2\,\lambda\,\frac12\,\frac{\alpha}{\alpha-d}\,
r_n^{d-\alpha}\log^2(r_n^\alpha),
\]
and since $\beta=1$ this is the asserted formula.

\smallskip\noindent
Now {assume $\beta>1$.}
By dominated convergence applied to \eqref{eq:K-log-form},
\[
K(x)\longrightarrow \int_1^\infty u^{-\beta}\log u\,\mathrm{d}u=\frac{1}{(\beta-1)^2},
\qquad x\downarrow 0.
\]
Therefore, again using $\int_0^1 t^{-d/\alpha}\,\mathrm{d}t=\alpha/(\alpha-d)$,
\[
\int_0^1 t^{-d/\alpha}K(at)\,\mathrm{d}t
\longrightarrow \frac{1}{(\beta-1)^2}\,\frac{\alpha}{\alpha-d}.
\]
Inserting into \eqref{eq:I1B-K} yields
\[
I_1^{(B)}\sim (2n)^d\,\omega_d\,r_n^{d}\,a\,\beta^2\,
\frac{1}{(\beta-1)^2}\,\frac{\alpha}{\alpha-d}
=(2n)^d\,\omega_d\,\beta^2\lambda\,\frac{\alpha}{\alpha-d}\,(\beta-1)^{-2}\,r_n^{d-\alpha}.
\]
This completes the proof.
\end{proof}

Recall that $\EE[F(n,r)]=I_1-I_2$, where $I_1$ counts exceedances between $B_n$ and $B_n^C$ and $I_2$ is the correction
counting pairs entirely contained in $B_n$.

\begin{lemma}\label{lem:I1-asymptotics}
We have, as $n\to\infty$,
\[
I_1 \sim (2n)^d\,\omega_d\,
\begin{cases}
\Gamma(1-\beta)\,\lambda^\beta\,\beta\,\dfrac{d}{\alpha\beta-d}\;
r_n^{\,d-\alpha\beta}\,\log(r_n^\alpha),
& 0<\beta<1,\\[2mm]
\dfrac12\,\lambda\,\dfrac{d}{\alpha-d}\;
r_n^{\,d-\alpha}\,\log^2(r_n^\alpha),
& \beta=1,\\[2mm]
\lambda\,\dfrac{\beta^2}{(\beta-1)^2}\,\dfrac{d}{\alpha-d}\;
r_n^{\,d-\alpha},
& \beta>1.
\end{cases}
\]
\end{lemma}

\begin{proof}
This is the combination of Lemmas~\ref{lem:I1A-asymptotics} and \ref{lem:I1B-asymptotics}, noting that
$I_1=-I_1^{(A)}+I_1^{(B)}$ and adding the leading constants in each regime.
\end{proof}

\begin{corollary}\label{cor:cn-choice}
Let $\theta:=\alpha\beta-d$ if $0<\beta<1$ and $\theta:=\alpha-d$ if $\beta\ge 1$. Define $c_n$ by
\[
c_n:=
\begin{cases}
\Bigl(2^d\omega_d\,\Gamma(1-\beta)\lambda^\beta\,\beta\,\dfrac{d}{\alpha\beta-d}\cdot \dfrac{\alpha d}{\theta}\Bigr)^{1/\theta}
\,n^{d/\theta}(\log n)^{1/\theta},
& 0<\beta<1,\\[2mm]
\Bigl(2^d\omega_d\,\dfrac12\lambda\,\dfrac{d}{\alpha-d}\cdot \dfrac{\alpha^2 d^2}{\theta^2}\Bigr)^{1/\theta}
\,n^{d/\theta}(\log n)^{2/\theta},
& \beta=1,\\[2mm]
\Bigl(2^d\omega_d\,\lambda\,\dfrac{\beta^2}{(\beta-1)^2}\,\dfrac{d}{\alpha-d}\Bigr)^{1/\theta}
\,n^{d/\theta},
& \beta>1,
\end{cases}
\]
and set $r_n=c_n r$ for fixed $r>0$. Then
\[
I_1 \sim r^{-\theta},
\qquad n\to\infty.
\]
\end{corollary}

\begin{proof}
Insert $r_n=c_n r$ into Lemma~\ref{lem:I1-asymptotics}. It suffices to check that the $n$--dependent factors cancel and
that the logarithmic terms satisfy the expected asymptotics.

If $0<\beta<1$, then $r_n^\theta=(c_n r)^\theta=c_n^\theta r^\theta$ and
\[
c_n^\theta = 2^d\omega_d\,\Gamma(1-\beta)\lambda^\beta\,\beta\,\frac{d}{\alpha\beta-d}\cdot \frac{\alpha d}{\theta}\;
n^d\log n.
\]
Moreover,
\[
\log(r_n^\alpha)=\alpha\log c_n + O(1)=\alpha\frac{d}{\theta}\log n + O(\log\log n),
\]
so $\log(r_n^\alpha)\sim \alpha(d/\theta)\log n$. Combining these relations yields, as $n\to \infty$, that
\[
(2n)^d\omega_d \cdot \Gamma(1-\beta)\lambda^\beta\,\beta\,\frac{d}{\alpha\beta-d}\;
r_n^{-\theta}\log(r_n^\alpha)\ \longrightarrow\ r^{-\theta}.
\]
The cases $\beta=1$ and $\beta>1$ are analogous: one uses
$\log(r_n^\alpha)\sim \alpha(d/\theta)\log n$ and $r_n^\theta\sim c_n^\theta r^\theta$, noting that the choice of $c_n$
was made so that the prefactor in Lemma~\ref{lem:I1-asymptotics} cancels. This proves the claim.
\end{proof}

\begin{theorem}\label{thm:mean-exceedances-limit}
Let $\theta:=\alpha\beta-d$ if $0<\beta<1$ and $\theta:=\alpha-d$ if $\beta\ge 1$, and assume that
\[
\theta\in(d,2d]\ \text{ if }\ \beta\le 1,
\qquad\text{and}\qquad
\theta\in(d,2d)\ \text{ if }\ \beta>1.
\]
Let $c_n$ be as in Corollary~\ref{cor:cn-choice} and set $r_n=c_n r$ for fixed $r>0$. Then $I_2=0$ for all $n$ large and
\[
\EE[F(n,r)] = I_1 - I_2 \longrightarrow r^{-\theta},
\qquad n\to\infty.
\]
\end{theorem}

\begin{proof}
By Corollary~\ref{cor:cn-choice}, $I_1\to r^{-\theta}$. It remains to show that $I_2=0$ for $n$ large. Under the stated
assumptions on $\theta$, the scaling $r_n=c_n r$ satisfies $r_n/n\to\infty$, hence $r_n\ge 2n$ for all sufficiently large $n$.
For such $n$ and any $x\in B_n$, we have $B_n\subset B_{r_n}(x)$, equivalently $B_n\cap B_{r_n}^C(x)=\varnothing$, so there
are no exceedance pairs entirely contained in $B_n$. Thus $I_2=0$ for all large $n$, and the limit follows.
\end{proof}

Finally the proof of Theorem \ref{thm:uncond-max} follows from Proposition \ref{prop:Poisson-approx-Fnr} combined with Theorem \ref{thm:mean-exceedances-limit}.

\section{Proof of Theorem \ref{thm:POT-hub}}\label{proofthm2}

Recall that we denote $W_n^*=\max_{x\in V\cap B_n}W_x$. We start this section with the following observation.

\begin{lemma}\label{le:TailW*n}
It holds
\[\mathbb{P}_0\left(W_n^*\leq d_n\right)=\left(1-d_n^{-\beta}\right)\exp\left(-(2n)^dd_n^{-\beta}\right)\sim\left(1-d_n^{-\beta}\right)^{(2n)^d+1}\]
\end{lemma}
\begin{proof}
We have
\begin{align*}
\mathbb{P}_0\left(W_n^*\leq d_n\right)&=\sum_{k=0}^\infty\mathbb{P}_0\left(W_n^*\leq d_n\vert \# \left(V\cap B_n\setminus\{0\}\right)=k\right)\mathbb{P}_0\left(\#\left(V\cap B_n\setminus\{0\}\right)=k\right)\\
&=\sum_{k=0}^\infty\mathbb{P}_0\left(W_0\leq d_n\right)^{k+1}\mathbb{P}_0\left(\#\left(V\cap B_n\setminus\{0\}\right)=k\right)\\
&=(1-d_n^{-\beta})\sum_{k=0}^\infty(1-d_n^{-\beta})^{k}\mathbb{P}_0\left(\#\left(V\cap B_n\setminus\{0\}\right)=k\right)\\
&=(1-d_n^{-\beta})\exp\left(-(2n)^dd_n^{-\beta}\right)
\end{align*}
since $\#\left(V\cap B_n\setminus\{0\}\right)$ is Poisson distributed with mean $(2n)^d$. The second claim follows from the classical approximation $\ln(1-d_n^{-\beta})\sim -d_n^{-\beta}$ (since $d_n\nearrow\infty$).
\end{proof}
\begin{remark}\label{re:maxweights}
It holds
\[\mathbb{P}_0\left(W_n^*\leq d_n\right)\longrightarrow\left\{\begin{matrix}
0&\mbox{if }n^{-d}d_n^{\beta}\rightarrow 0\\
1&\mbox{if }n^{-d}d_n^{\beta}\rightarrow \infty\\
\exp ( - \frac{2^d}{c}) &\mbox{if }n^{-d}d_n^{\beta}\rightarrow c \in (0,\infty)
\end{matrix}\right.\]
\end{remark}

\subsection{Proof of Lemma \ref{lem:POT-rooted-unrooted}}

We first consider 
$$ \mathbb{P}_0\left(e^*_0>t_n\vert W_0>d_n\right) .$$
We have
\begin{align*}
\mathbb{P}_0\left(e^*_0>t_n\vert W_0>d_n\right) 
&= \mathbb{P}_0\left(e^*_0>t_n,\, W_0>d_n\right) \mathbb{P}_0\left(W_0>d_n\right)^{-1}\\
&= \int_{d_n}^\infty \mathbb{P}_0\left(e^*_0>t_n\vert W_0=w\right) \, \d F_{W}(w)d_n^\beta\\
&\color{black}= \int_{d_n}^\infty 1 - \exp\left(-\mathbb{E}_0\left(\int_{B_{t_n}^c(0)}\left(1-\exp\left(-\frac{\lambda W_yw}{\vert y\vert^\alpha}\right)\right)\, \d y\right)\right) \, \d F_{W_0}(w)d_n^\beta\\
& =: I_n.
%&\color{black} = (I_1 + I_2)d_n^\beta,
%\end{align*}
%where $I_1$ is the corresponding integral from $d_n$ to $f_n$, and $I_2$ is the corresponding integral from $f_n$ to $\infty$. We show that $I_2$ is a negligible part:
%\begin{align*}
%I_2 &= \int_{f_n}^\infty 1 - \exp\left(-\mathbb{E}_0\left(\int_{B_{t_n}^c(0)}\left(1-\exp\left(-\frac{\lambda W_yw}{\vert y\vert^\alpha}\right)\right)\, \d y\right)\right) \, \d F_{W_0}(w)\\
%&\color{black} = O \Big( t_n^{d-\alpha} \int_{f_n}^\infty w \d F_{W_0}(w)\Big) = O (t_n^{d-\alpha} f_n^{1-\beta} ).
%\end{align*}
%On the other hand, using that $\frac{f_n}{t^{\alpha}_n} \to 0$, we have
%\begin{align*}
%I_1 &= \int_{d_n}^{f_n} 1 - \exp\left(-\mathbb{E}_0\left(\int_{B_{t_n}^c(0)}\left(1-\exp\left(-\frac{\lambda W_yw}{\vert y\vert^\alpha}\right)\right)\, \d y\right)\right) \, \d F_{W_0}(w)\\
%& \color{black}\sim \int_{d_n}^{\infty} \mathbb{E}_0\left(\int_{B_{t_n}^c(0)}\frac{\lambda W_yw}{\vert y\vert^\alpha} \, \d y \right)  \, \d F_{W_0}(w) d_n^\beta\\
%&= \lambda\mathbb{E}_0\left(W\right) \int_{B_{t_n}^c(0)}\vert y\vert^{-\alpha} \, \d y \int_{d_n}^{\infty} w \, \d F_{W_0}(w) d_n^\beta\\
%&= \lambda\frac{\beta}{\beta-1}d\omega_d \int_{t_n}^\infty s^{d-\alpha-1}\, \d s \int_{d_n}^{\infty} w \, w^{-\beta-1}\, \d w d_n^\beta\\
%&= \frac{\lambda\beta d\omega_d}{(\beta-1)(\alpha-d)}t_n^{d-\alpha}\frac{1}{\beta-1} d_n.\\
%&% \sim \frac{\lambda\beta d\omega_d}{(\beta-1)(\alpha-d)}t_n^{d-\alpha}\frac{1}{\beta-1} d_n{},
\end{align*}

The next lemma derives the precise asymptotics of the expression $I_n$.

\begin{lemma}\label{lem:In-limit-allbeta-gpd}
Fix $t>0$ and let $c_n\to\infty$. Define
\[
\theta=
\begin{cases}
\alpha\beta-d, & 0<\beta<1,\\
\alpha-d, & \beta\ge 1,
\end{cases}
\qquad
\gamma:=\min\{1,\beta\},
\qquad
t_n=t_n(t):=c_n^{1/\theta}\Bigl(1+\frac{t}{\theta}\Bigr).
\]
Set
\[
d_n:=
\begin{cases}
\kappa_\beta^{-1}\,c_n^{1/\beta}, & 0<\beta<1,\\[1mm]
\kappa_1^{-1}\,c_n/\log c_n, & \beta=1,\\[1mm]
\kappa_\beta^{-1}\,c_n, & \beta>1,
\end{cases}
\]
where the (explicit) constants $\kappa_\beta\in(0,\infty)$ are given by
\[
\kappa_\beta:=
\begin{cases}
\bigl(\Gamma(1-\beta)^2\,\lambda^\beta\,\dfrac{d\omega_d}{\alpha\beta-d}\bigr)^{1/\beta}, & 0<\beta<1,\\[2mm]
\lambda\,\dfrac{d\omega_d}{\alpha-d}\,\dfrac{d}{\alpha-d}, & \beta=1,\\[2mm]
\lambda\,\dfrac{d\omega_d}{\alpha-d}\,\dfrac{\beta^2}{(\beta-1)^2}, & \beta>1.
\end{cases}
\]
For $w\ge 1$ define
\[
\Phi_n(w):=\EE\Bigg[\int_{|y|>t_n}\Bigl(1-\exp\Bigl\{-\lambda w W/|y|^\alpha\Bigr\}\Bigr)\,\mathrm{d}y\Bigg],
\]
and set
\begin{equation}\label{eq:In-def-slv}
I_n=\beta d_n^\beta\int_{d_n}^{\infty}
\Phi_n(w)\Bigl(\int_{0}^{1}e^{-s\Phi_n(w)}\,\mathrm{d}s\Bigr)\,w^{-\beta-1}\,\mathrm{d}w.
\end{equation}
Define, for $t>0$,
\[
a_\beta(t):=
\begin{cases}
\Gamma(1-\beta)^{-1}\Bigl(1+\dfrac{t}{\theta}\Bigr)^{-\theta}, & 0<\beta<1,\\[2mm]
\Bigl(1+\dfrac{t}{\theta}\Bigr)^{-\theta}, & \beta=1,\\[2mm]
\dfrac{\beta-1}{\beta}\Bigl(1+\dfrac{t}{\theta}\Bigr)^{-\theta}, & \beta>1,
\end{cases}
\]
and
\begin{equation*}\label{eq:Gbeta-def-slv}
G_\beta(t):=\beta\int_{1}^{\infty}\bigl(1-e^{-a_\beta(t)v^\gamma}\bigr)\,v^{-\beta-1}\,\mathrm{d}v.
\end{equation*}
Then, as $n\to\infty$,
\[
I_n\ \longrightarrow\ G_\beta(t).
\]
Moreover, as $t\to\infty$, $G_\beta(t)\downarrow 0$ and
\[
G_\beta(t)\ \sim\
\begin{cases}
\Gamma(1-\beta)\,a_\beta(t), & 0<\beta<1,\\[1mm]
a_1(t)\,\log\!\bigl(1/a_1(t)\bigr), & \beta=1,\\[1mm]
\dfrac{\beta}{\beta-1}\,a_\beta(t), & \beta>1.
\end{cases}
\]
Moreover, as $t\to\infty$,
\[
G_\beta(t)\ \sim\ \Bigl(1+\frac{t}{\theta}\Bigr)^{-\theta},
\qquad \beta\neq 1,
\]
whereas in the borderline case $\beta=1$,
\[
G_1(t)\ \sim\ \theta\Bigl(1+\frac{t}{\theta}\Bigr)^{-\theta}\log\!\Bigl(1+\frac{t}{\theta}\Bigr).
\]

\end{lemma}

\begin{proof}
\emph{Step 1: rescaling in \eqref{eq:In-def-slv}.}
Substitute $w=d_nv$ with $v\in[1,\infty)$. Since $\mathrm{d}w=d_n\,\mathrm{d}v$ and
$w^{-\beta-1}=(d_nv)^{-\beta-1}$, the factor $d_n^\beta$ cancels and we obtain
\begin{equation}\label{eq:In-rescaled-slv}
I_n=\beta\int_{1}^{\infty}
\Phi_n(d_nv)\Bigl(\int_{0}^{1}e^{-s\Phi_n(d_nv)}\,\mathrm{d}s\Bigr)\,v^{-\beta-1}\,\mathrm{d}v.
\end{equation}

\medskip
\emph{Step 2: preliminary facts on $\psi_\beta(x):=\EE[1-e^{-xW}]$.}
By the Pareto density and Tonelli,
\[
\psi_\beta(x)=\beta\int_{1}^{\infty}(1-e^{-xu})u^{-\beta-1}\,\mathrm{d}u,\qquad x\ge 0.
\]
We record the small--$x$ asymptotics and convenient bounds.

\smallskip
\noindent\textit{(i) Case $0<\beta<1$.}
With the substitution $s=xu$,
\[
\psi_\beta(x)
=\beta x^\beta\int_{x}^{\infty}(1-e^{-s})s^{-\beta-1}\,\mathrm{d}s.
\]
Since $\int_{0}^{\infty}(1-e^{-s})s^{-\beta-1}\,\mathrm{d}s=\Gamma(1-\beta)/\beta$ for $0<\beta<1$, we obtain
\begin{equation}\label{eq:psi-beta<1-slv}
\psi_\beta(x)\sim \Gamma(1-\beta)x^\beta,\qquad x\downarrow 0.
\end{equation}
Moreover, for all $x\in(0,1]$,
\begin{equation}\label{eq:psi-beta<1-bound-slv}
\psi_\beta(x)\le C x^\beta,
\end{equation}
since the integral above is bounded uniformly in $x\in(0,1]$.

\smallskip
\noindent\textit{(ii) Case $\beta=1$.}
Again substituting $s=xu$ gives
\[
\psi_1(x)=x\int_{x}^{\infty}(1-e^{-s})s^{-2}\,\mathrm{d}s.
\]
Integrating by parts,
\[
\int_{x}^{\infty}(1-e^{-s})s^{-2}\,\mathrm{d}s
=\frac{1-e^{-x}}{x}+\int_{x}^{\infty}e^{-s}s^{-1}\,\mathrm{d}s,
\]
hence
\begin{equation}\label{eq:psi-beta=1-slv}
\psi_1(x)=x\log(1/x)+O(x),\qquad x\downarrow 0.
\end{equation}
Moreover, for all $x\in(0,1]$,
\begin{equation}\label{eq:psi-beta=1-bound-slv}
\psi_1(x)\le C x\bigl(1+\log(1/x)\bigr).
\end{equation}

\smallskip
\noindent\textit{(iii) Case $\beta>1$.}
Since $\EE[W]=\beta/(\beta-1)<\infty$ and $0\le (1-e^{-xW})/(xW)\le 1$,
dominated convergence yields
\begin{equation}\label{eq:psi-beta>1-slv}
\psi_\beta(x)\sim x\EE[W]=x\,\frac{\beta}{\beta-1},\qquad x\downarrow 0,
\end{equation}
and the bound
\begin{equation}\label{eq:psi-beta>1-bound-slv}
\psi_\beta(x)\le Cx,\qquad x\in(0,1].
\end{equation}

\medskip
\emph{Step 3: pointwise limit of $\Phi_n(d_nv)$.}
By Tonelli,
\begin{equation*}\label{eq:Phi-repr-slv}
\Phi_n(d_nv)=\int_{|y|>t_n}\psi_\beta\!\Bigl(\frac{\lambda d_nv}{|y|^\alpha}\Bigr)\,\mathrm{d}y.
\end{equation*}
Fix $v\ge 1$. We first show that the argument of $\psi_\beta$ is uniformly small on $\{|y|>t_n\}$:
\begin{equation}\label{eq:arg-small-slv}
\sup_{|y|>t_n}\frac{\lambda d_nv}{|y|^\alpha}\le \frac{\lambda d_nv}{t_n^\alpha}\longrightarrow 0,
\qquad n\to\infty.
\end{equation}
Indeed, $t_n^\alpha=c_n^{\alpha/\theta}(1+t/\theta)^\alpha$, while $d_n$ grows at most polynomially in $c_n$:
for $\beta>1$, $d_n\asymp c_n$; for $\beta=1$, $d_n\asymp c_n/\log c_n$; for $0<\beta<1$, $d_n\asymp c_n^{1/\beta}$.
In all cases $\alpha/\theta>1$, hence $d_n/t_n^\alpha\to 0$.

\smallskip
We now distinguish the three cases.

\smallskip
\noindent\textit{Case $0<\beta<1$.}
From \eqref{eq:psi-beta<1-slv} and \eqref{eq:arg-small-slv}, uniformly for $|y|>t_n$,
\[
\psi_\beta\!\Bigl(\frac{\lambda d_nv}{|y|^\alpha}\Bigr)
\sim \Gamma(1-\beta)\lambda^\beta d_n^\beta v^\beta |y|^{-\alpha\beta}.
\]
Dominated convergence applies using \eqref{eq:psi-beta<1-bound-slv} and \eqref{eq:arg-small-slv}, yielding
\[
\Phi_n(d_nv)\sim \Gamma(1-\beta)\lambda^\beta d_n^\beta v^\beta
\int_{|y|>t_n}|y|^{-\alpha\beta}\,\mathrm{d}y.
\]
Evaluating the spatial integral in polar coordinates,
\[
\int_{|y|>t_n}|y|^{-\alpha\beta}\,\mathrm{d}y
=d\omega_d\int_{t_n}^{\infty} r^{d-\alpha\beta-1}\,\mathrm{d}r
=\frac{d\omega_d}{\alpha\beta-d}\,t_n^{d-\alpha\beta}
=\frac{d\omega_d}{\theta}\,t_n^{-\theta}.
\]
Since $t_n^{-\theta}=c_n^{-1}(1+t/\theta)^{-\theta}$ and $d_n^\beta=\kappa_\beta^{-\beta}c_n$, we obtain
\[
\Phi_n(d_nv)\longrightarrow
\Gamma(1-\beta)\lambda^\beta\frac{d\omega_d}{\theta}\,\kappa_\beta^{-\beta}\,
\Bigl(1+\frac{t}{\theta}\Bigr)^{-\theta}v^\beta.
\]
With $\kappa_\beta^\beta=\Gamma(1-\beta)^2\lambda^\beta\,\dfrac{d\omega_d}{\theta}$ this limit equals
$\Gamma(1-\beta)^{-1}(1+t/\theta)^{-\theta}v^\beta=a_\beta(t)v^\beta$.

\smallskip
\noindent\textit{Case $\beta=1$.}
Write $x=\lambda d_nv/|y|^\alpha$. By \eqref{eq:psi-beta=1-slv} and \eqref{eq:arg-small-slv},
\[
\psi_1(x)=x\log(1/x)\,(1+o(1))
\qquad\text{uniformly for }|y|>t_n,
\]
and dominated convergence using \eqref{eq:psi-beta=1-bound-slv} gives
\begin{equation}\label{eq:Phi-beta=1-main-slv}
\Phi_n(d_nv)
= \lambda d_n v\int_{|y|>t_n}|y|^{-\alpha}\log\!\Bigl(\frac{|y|^\alpha}{\lambda d_nv}\Bigr)\,\mathrm{d}y\,(1+o(1)).
\end{equation}
Split the logarithm as
\[
\log\!\Bigl(\frac{|y|^\alpha}{\lambda d_nv}\Bigr)
=\log\!\Bigl(\frac{|y|^\alpha}{t_n^\alpha}\Bigr)+\log\!\Bigl(\frac{t_n^\alpha}{\lambda d_nv}\Bigr).
\]
The first term contributes $O(\lambda d_n v\,t_n^{d-\alpha})$ since
$\int_{|y|>t_n}|y|^{-\alpha}\log(|y|^\alpha/t_n^\alpha)\,\mathrm{d}y=O(t_n^{d-\alpha})$.
For the second term,
\[
\int_{|y|>t_n}|y|^{-\alpha}\,\mathrm{d}y
=\frac{d\omega_d}{\alpha-d}\,t_n^{d-\alpha}
=\frac{d\omega_d}{\theta}\,t_n^{-\theta}
=\frac{d\omega_d}{\theta}\,c_n^{-1}\Bigl(1+\frac{t}{\theta}\Bigr)^{-\theta}.
\]
Moreover, since $t_n^\alpha=c_n^{\alpha/\theta}(1+t/\theta)^\alpha$ and $d_n=\kappa_1^{-1}c_n/\log c_n$,
\[
\log\!\Bigl(\frac{t_n^\alpha}{d_n}\Bigr)
=\Bigl(\frac{\alpha}{\theta}-1\Bigr)\log c_n + \log\log c_n + O(1)
=\frac{d}{\theta}\log c_n + o(\log c_n),
\]
and hence
\[
\frac{1}{\log c_n}\log\!\Bigl(\frac{t_n^\alpha}{\lambda d_nv}\Bigr)\ \longrightarrow\ \frac{d}{\theta},
\qquad n\to\infty,
\]
for each fixed $v\ge 1$. Combining this with \eqref{eq:Phi-beta=1-main-slv} yields
\[
\Phi_n(d_nv)\ \longrightarrow\
\lambda\,\kappa_1^{-1}\,\frac{d}{\theta}\,\frac{d\omega_d}{\theta}\,
v\Bigl(1+\frac{t}{\theta}\Bigr)^{-\theta}.
\]
With $\kappa_1=\lambda\,\dfrac{d\omega_d}{\theta}\,\dfrac{d}{\theta}$ we conclude that
$\Phi_n(d_nv)\to (1+t/\theta)^{-\theta}v=a_1(t)v$.

\smallskip
\noindent\textit{Case $\beta>1$.}
Fix $v\ge 1$ and write $w=d_nv$. Since $\sup_{|y|>t_n}\lambda w/|y|^\alpha\le \lambda d_nv/t_n^\alpha\to 0$,
dominated convergence and \eqref{eq:psi-beta>1-slv} yield, uniformly for $|y|>t_n$,
\[
\EE\Bigl[1-\exp\Bigl\{-\lambda w W/|y|^\alpha\Bigr\}\Bigr]
\sim \frac{\lambda w}{|y|^\alpha}\EE[W]
=\frac{\lambda w}{|y|^\alpha}\,\frac{\beta}{\beta-1}.
\]
Consequently,
\[
\Phi_n(d_nv)\sim \lambda d_n v\,\frac{\beta}{\beta-1}\int_{|y|>t_n}|y|^{-\alpha}\,\mathrm{d}y
=\lambda d_n v\,\frac{\beta}{\beta-1}\,\frac{d\omega_d}{\theta}\,t_n^{-\theta}.
\]
With $t_n^{-\theta}=c_n^{-1}(1+t/\theta)^{-\theta}$ and $d_n=\kappa_\beta^{-1}c_n$ we obtain
\[
\Phi_n(d_nv)\longrightarrow
\lambda\,\frac{\beta}{\beta-1}\,\frac{d\omega_d}{\theta}\,\kappa_\beta^{-1}\,
\Bigl(1+\frac{t}{\theta}\Bigr)^{-\theta}v.
\]
With $\kappa_\beta=\lambda\,\dfrac{d\omega_d}{\theta}\,\dfrac{\beta^2}{(\beta-1)^2}$ this limit equals
$\frac{\beta-1}{\beta}(1+t/\theta)^{-\theta}v=a_\beta(t)v$.

\medskip
\emph{Step 4: pointwise limit of the full integrand.}
For $x\ge 0$ we have $x\int_0^1 e^{-sx}\,\mathrm{d}s=1-e^{-x}$. Hence, for fixed $v\ge 1$,
\[
\Phi_n(d_nv)\Bigl(\int_{0}^{1}e^{-s\Phi_n(d_nv)}\,\mathrm{d}s\Bigr)
\longrightarrow 1-e^{-a_\beta(t)v^\gamma}.
\]

\medskip
\emph{Step 5: domination and convergence of $I_n$.}
Using $1-e^{-x}\le \min\{x,1\}$ and the bounds \eqref{eq:psi-beta<1-bound-slv},
\eqref{eq:psi-beta=1-bound-slv}, \eqref{eq:psi-beta>1-bound-slv} together with \eqref{eq:arg-small-slv},
we obtain for all $n$ large and all $v\ge 1$,
\[
0\le \Phi_n(d_nv)\le C v^\gamma
\]
with $C<\infty$ depending only on $(d,\alpha,\beta,\lambda,t)$. Therefore the integrand in
\eqref{eq:In-rescaled-slv} is dominated by $\min\{Cv^\gamma,1\}v^{-\beta-1}$, which is integrable on $[1,\infty)$.
Dominated convergence applied to \eqref{eq:In-rescaled-slv} yields $I_n\to G_\beta(t)$.

\medskip
\emph{Step 6: asymptotics of $G_\beta(t)$ as $t\to\infty$.}
As $t\to\infty$, $a_\beta(t)\downarrow 0$. For $\beta>1$, dominated convergence gives
\[
G_\beta(t)=\beta\int_1^\infty \bigl(1-e^{-a_\beta(t)v}\bigr)v^{-\beta-1}\,\mathrm{d}v
\sim \beta a_\beta(t)\int_1^\infty v^{-\beta}\,\mathrm{d}v
=\frac{\beta}{\beta-1}\,a_\beta(t).
\]
For $0<\beta<1$, substitute $u=a_\beta(t)^{1/\beta}v$ to obtain
\[
G_\beta(t)
=\beta a_\beta(t)\int_{a_\beta(t)}^\infty \bigl(1-e^{-u}\bigr)u^{-\beta-1}\,\mathrm{d}u
\longrightarrow \beta a_\beta(t)\int_{0}^\infty \bigl(1-e^{-u}\bigr)u^{-\beta-1}\,\mathrm{d}u
=\Gamma(1-\beta)\,a_\beta(t),
\]
where $\int_{0}^\infty (1-e^{-u})u^{-\beta-1}\,\mathrm{d}u=\Gamma(1-\beta)/\beta$.
Finally, for $\beta=1$, writing $A=a_1(t)$ and splitting at $v=1/A$,
\[
G_1(t)=\int_1^\infty \bigl(1-e^{-Av}\bigr)v^{-2}\,\mathrm{d}v
=\int_1^{1/A}\bigl(1-e^{-Av}\bigr)v^{-2}\,\mathrm{d}v+\int_{1/A}^\infty \bigl(1-e^{-Av}\bigr)v^{-2}\,\mathrm{d}v,
\]
and $1-e^{-Av}\sim Av$ on $[1,1/A]$ while $1-e^{-Av}\le 1$ on $[1/A,\infty)$ yields
\[
G_1(t)\sim A\int_1^{1/A}\frac{\mathrm{d}v}{v}=A\log(1/A),
\qquad A\downarrow 0.
\]
Inserting the definition of $a_\beta(t)$ completes the proof.
\end{proof}

%

%\subsubsection{$\mathbb{P}_0\left(e^*_0>t_n|W_0>d_n\right)$, $W_0\in RV_{-\beta}$}
%\begin{lemma}%\label{le:POTeOWo}
%For some increasing sequence $(c_n)$, let $t_n=c_n^{\frac{1}{\alpha -d}}(1+\frac{t}{\alpha-d})$, $t>0$, and $d_n=\kappa^{-1}c_n$ with $\kappa=\lambda\beta d\omega_d(\beta-1)^{-2}(\alpha-d)^{-1}$. Then, 
%\[\mathbb{P}_0\left((\alpha -d)\frac{ e^*_0 - c_n^{\frac{1}{\alpha -d}}}{c_n^{\frac{1}{\alpha -d}}}>t\,\Big\vert \, W_0>\kappa^{-1}c_n\right) = \mathbb{P}_0\left(e^*_0>t_n\vert W_0>d_n\right)\rightarrow  \left(1+\frac{t}{\alpha-d}\right)^{d-\alpha}.\]
%\end{lemma}

It remains to consider the probability

$$\mathbb{P}_0\left(e^*_n>t_n|W_0>d_n\right).$$
Define, for $x \in V \cap B_n$,
    \[
e_x^* \,:=\, \sup\bigl\{|y|:\ y\in V\setminus\{x\},\ x\leftrightarrow y\bigr\}, \quad \bar{e}_n(x) := \max_{ y\in V\setminus\{x\}} e_y^*, 
\]
We have: 
\begin{align*}
\mathbb{P}_0\left(e^*_n>t_n\vert W_0>d_n\right)&=\mathbb{P}_0\left(e^*_n>t_n, e^*_0<t_n\vert W_0>d_n\right)+\mathbb{P}_0\left(e^*_n>t_n, e^*_0>t_n\vert W_0>d_n\right)\\
&=\mathbb{P}_0\left(\overline{e}^*_n(0)>t_n,e^*_0<t_n\vert W_0>d_n\right)+\mathbb{P}_0\left(e^*_0>t_n\vert W_0>d_n\right).
\end{align*}
The second sumand in the r.h.s.$\,$ of the last expression converges to the desired quantity by Lemma \ref{lem:In-limit-allbeta-gpd}. Let us justify that the first sumand vanishes under our assumptions. It suffices to write that: 
\begin{align*}
\mathbb{P}_0\left(\overline{e}^*_n(0)>t_n,e^*_0<t_n\vert W_0>d_n\right)&\leq\mathbb{P}_0\left(\overline{e}^*_n(0)>t_n\vert W_0>d_n\right)\\
&=\mathbb{P}_0\left(\overline{e}^*_n(0)>t_n\right) \underset{n\to\infty}{\longrightarrow} 0
\end{align*}
by Theorem \ref{thm:uncond-max}, since $t_n/\max(n,n^\frac{d}{\theta} (\log n)^{\frac{2}{\theta}})\rightarrow \infty$ by assumption, and the fact that the edge probability is decreasing in $\alpha$ and $\beta$.

\subsection{Proof of Lemma \ref{lem:rooted-exceed-asymptotics}}

%\begin{lemma}\label{le:POTeOWstar}
%Assume that $t_n\geq Kn$ for some $K$ not too small and $d_n=o\left(t_n^{\alpha}\right)\nearrow\infty$. Then, 
%\[\mathbb{P}_0\left(e^*_0>t_n\vert W^*_n>d_n\right)\sim\kappa t_n^{d-\alpha},\]
%where $\kappa=\lambda\beta d\omega_d(\beta-1)^{-2}(\alpha-d)^{-1}$.
%\end{lemma}
First let us rewrite
\begin{align*}
\mathbb{P}_0\left(e^*_0>t_n\vert W^*_n>d_n\right)&=\frac{\mathbb{P}_0\left(e^*_0>t_n, W^*_n>d_n\right)}{\mathbb{P}_0\left(W^*_n>d_n\right)}\\
&=\frac{\mathbb{P}_0\left(e^*_0>t_n, W^*_n>d_n\right)}{1-\left(1-d_n^{-\beta}\right)\exp\left(-(2n))^dd_n^{-\beta}\right)} ,
\end{align*}
but
\begin{align*}
\mathbb{P}_0&\left(e^*_0>t_n, W^*_n>d_n\right)\\
&=\int_1^{d_n}\mathbb{P}_0\left(e^*_0>t_n, \overline{W}^*_n>d_n\vert W_0=w\right)\d F_{W_0}(w)+\int_{d_n}^\infty\mathbb{P}_0\left(e^*_0>t_n\vert W_0=w\right)\d F_{W_0}(w) \\
& = \int_1^{d_n}\mathbb{P}_0\left(e^*_0>t_n, \overline{W}^*_n>d_n\vert W_0=w\right)\d F_{W_0}(w)+\mathbb{P}_0\left(e^*_0>t_n\vert W_0>d_n\right),
\end{align*}
where
$$ \overline{W}^*_n:= \max_{x \in V \cap B_n \setminus \{ 0 \} } W_x, $$
and that can be rewritten due to independence, if $t_n\geq Kn$, as $I_1+I_2$ with : 
\[I_1=\mathbb{P}_0\left(\overline{W}^*_n>d_n\right)\int_1^{d_n}\mathbb{P}_0\left(e^*_0>t_n\vert W_0=w\right)\d F_{W_0}(w)\]
and 
\[I_2= \mathbb{P}_0\left(e^*_0>t_n\vert W_0>d_n\right) = \int_{d_n}^\infty 1 - \exp\left(-\mathbb{E}_0\left(\int_{B_{t_n}^c(0)}\left(1-\exp\left(-\frac{\lambda W_yw}{\vert y\vert^\alpha}\right)\right)\, \d y\right)\right) \, \d F_{W_0}(w) .\]

First we consider the integral in the definition of $I_1$. The proof of the following lemma is carried out essentially as in the previous steps.

\begin{lemma}\label{lem:lowweight-int-asymp}
Let $t_n\to\infty$ and $d_n\to\infty$ satisfy $t_n\ge Kn$ for some $K>0$ and all large $n$, and assume
\begin{equation*}
\frac{d_n}{t_n^\alpha}\longrightarrow 0.
\end{equation*}
Define
\[
A_n:=\int_{1}^{d_n}\PP_0\!\left(e_0^*>t_n\,\big|\,W_0=w\right)\,\mathrm{d}F_W(w)
=\beta\int_{1}^{d_n}\bigl(1-e^{-\Phi_n(w)}\bigr)\,w^{-\beta-1}\,\mathrm{d}w,
\]
with
\[
\Phi_n(w):=\EE\Bigg[\int_{|y|>t_n}\Bigl(1-\exp\Bigl\{-\lambda w W/|y|^\alpha\Bigr\}\Bigr)\,\mathrm{d}y\Bigg].
\]
Then
\[
A_n\sim \beta\int_{1}^{d_n}\Phi_n(w)\,w^{-\beta-1}\,\mathrm{d}w.
\]
Moreover, if
\[
d_n=
\begin{cases}
o\!\left(t_n^{\alpha\beta}\right), & 0<\beta<1,\\[1mm]
o\!\left(t_n^\alpha/\log t_n\right), & \beta=1,\\[1mm]
o\!\left(t_n^\alpha\right), & \beta>1,
\end{cases}
\]
then, as $n\to\infty$,
\begin{equation*}\label{eq:An-asymp-cases}
A_n\ \sim\
\begin{cases}
\kappa_{<}\,t_n^{d-\alpha\beta}\,\log d_n, & 0<\beta<1,\\[1mm]
\kappa_{=}\,t_n^{d-\alpha}\Bigl(\log(t_n^\alpha)\,\log d_n-\tfrac12(\log d_n)^2\Bigr), & \beta=1,\\[1mm]
\kappa_{>}\,t_n^{d-\alpha}\bigl(1-d_n^{1-\beta}\bigr), & \beta>1,
\end{cases}
\end{equation*}
where
\[
\kappa_{<}:=\beta\,\Gamma(1-\beta)\lambda^\beta\,\frac{d\omega_d}{\alpha\beta-d},
\qquad
\kappa_{=}:=\lambda\,\frac{d\omega_d}{\alpha-d},
\qquad
\kappa_{>}:=\lambda\,\frac{d\omega_d}{\alpha-d}\,\frac{\beta^2}{\beta-1}.
\]
\end{lemma}

Then, we collect terms and conclude, using Lemma \ref{le:TailW*n}, Lemma \ref{lem:In-limit-allbeta-gpd}:

\begin{align*}
\mathbb{P}_0\left(e^*_0>t_n\vert W^*_n>d_n\right)
&=\frac{\mathbb{P}_0\left(e^*_0>t_n, W^*_n>d_n\right)}{1-\left(1-d_n^{-\beta}\right)\exp\left(-(2n))^dd_n^{-\beta}\right)}\\
& = \frac{\mathbb{P}_0\left(\overline{W}^*_n>d_n\right)A_n+\mathbb{P}_0\left(e^*_0>t_n\vert W_0>d_n\right)}{1-\left(1-d_n^{-\beta}\right)\exp\left(-(2n))^dd_n^{-\beta}\right)}\\
& = \frac{\mathbb{P}_0\left(\overline{W}^*_n>d_n\right)A_n+O(d_n^{-\beta})}{1-\left(1-d_n^{-\beta}\right)\exp\left(-(2n))^dd_n^{-\beta}\right)}.
\end{align*}
Now, the proof of Lemma \ref{lem:rooted-exceed-asymptotics} follows from combining Lemma \ref{lem:lowweight-int-asymp}, Lemma \ref{le:TailW*n} and the definition of the sequences $t_n$ and $d_n$.

\subsection{Proof of Lemma \ref{le:jointedgesweights}}

    We will apply \cite[Theorem 3.1]{Penrose2} to prove the claim. \newline
    \emph{Step 1.} For $x,y \in \mathbb{R}$ we define the two coupled variables $U_{x,y}$ and $V_{x,y}$ by
    \begin{align*}
        U_{x,y} & = \sum_{z_1\in \mathcal{P} \cap B_n} \sum_{z_2\in \mathcal{P}\cap B^C_{t_n} (z_1)  } \mathbf{1}_{\{ {W_{z_1} > d_n}  \}}\mathbf{1}_{\{ {z_1\leftrightarrow z_2} \}} ,\\
        V_{x,y} & = \sum_{z_1\in (\mathcal{P}\cup \{x,y \}) \cap B_n} \sum_{z_2\in (\mathcal{P}\cup \{x,y \})\cap B^C_{t_n} (z_1)  } \mathbf{1}_{\{ {W_{z_1} > d_n}  \}}\mathbf{1}_{\{ {z_1\leftrightarrow z_2} \}} - \mathbf{1}_{\{ {W_{x} > d_n}  \}}\mathbf{1}_{\{ {x\leftrightarrow y} \}} ,
    \end{align*}
    that is, $V_{x,y}$ is the number of exceedances in $B_n$, other than the one at $x,y$ (if there is one), in the graph when adding two vertices $x,y$, and $U_{x,y}$ is the number of exceedances in the induced graph.
    \newline

\emph{Step 2.} For $x,y\in\RR^d$ define
\[
p(x,y)
:=\EE\!\left[\mathbf{1}_{\{W_x>d_n\}}\mathbf{1}_{\{x\leftrightarrow y\}}\right]
\mathbf{1}_{\{x\in B_n\}}\mathbf{1}_{\{y\in B^c_{t_n}(x)\}} .
\]
As before, conditioning on $W_x$ and using independence of the weights, we obtain the exact identity
\begin{align}
p(x,y)
&=\mathbf{1}_{\{x\in B_n\}}\mathbf{1}_{\{|x-y|>t_n\}}
\int_{d_n}^{\infty}
\EE\!\left[1-\exp\!\left\{-\frac{\lambda w W_y}{|x-y|^\alpha}\right\}\right]\,
\mathrm{d}F_W(w) \notag\\
&=\mathbf{1}_{\{x\in B_n\}}\mathbf{1}_{\{|x-y|>t_n\}}
\int_{d_n}^{\infty}\psi_\beta\!\left(\frac{\lambda w}{|x-y|^\alpha}\right)\,
\mathrm{d}F_W(w),\label{eq:pxy-psi}
\end{align}
where
\[
\psi_\beta(u):=\EE\!\left[1-e^{-uW}\right],\qquad u\ge 0.
\]
We next bound \eqref{eq:pxy-psi} using the regime--dependent estimates for $\psi_\beta$ from
Lemma~\ref{lem:In-limit-allbeta-gpd} (Step~2 therein). Write $r:=|x-y|$ and note that on
$\{|x-y|>t_n\}$ we have $r\ge t_n$.

\smallskip\noindent
\emph{Case $0<\beta<1$.}
Using $\psi_\beta(u)\le C u^\beta$ for $u\in(0,1]$ and $\psi_\beta(u)\le 1$ for all $u\ge 0$, we split the
$w$--integral at $r^\alpha/\lambda$ to obtain
\begin{align*}
\int_{d_n}^{\infty}\psi_\beta\!\left(\frac{\lambda w}{r^\alpha}\right)\,\mathrm{d}F_W(w)
&\le C\int_{d_n}^{r^\alpha/\lambda}\Bigl(\frac{\lambda w}{r^\alpha}\Bigr)^\beta\,\beta w^{-\beta-1}\,\mathrm{d}w
+\int_{r^\alpha/\lambda}^{\infty}\beta w^{-\beta-1}\,\mathrm{d}w\\
&\le C r^{-\alpha\beta}\int_{d_n}^{r^\alpha/\lambda}\frac{\mathrm{d}w}{w}
+ C r^{-\alpha\beta}
\ \le\ C r^{-\alpha\beta}\Bigl(1+\log\!\frac{r^\alpha}{d_n}\Bigr).
\end{align*}
Consequently,
\begin{equation}\label{eq:pxy-bound-beta<1}
p(x,y)\ \le\ C\,\mathbf{1}_{\{x\in B_n\}}\mathbf{1}_{\{|x-y|>t_n\}}\,
|x-y|^{-\alpha\beta}\Bigl(1+\log\!\frac{|x-y|^\alpha}{d_n}\Bigr).
\end{equation}

\smallskip\noindent
\emph{Case $\beta=1$.}
Using $\psi_1(u)\le C u\bigl(1+\log(1/u)\bigr)$ for $u\in(0,1]$ and again splitting at $r^\alpha/\lambda$ yields
\begin{align*}
\int_{d_n}^{\infty}\psi_1\!\left(\frac{\lambda w}{r^\alpha}\right)\,\mathrm{d}F_W(w)
&\le C\int_{d_n}^{r^\alpha/\lambda}\frac{\lambda w}{r^\alpha}\Bigl(1+\log\!\frac{r^\alpha}{\lambda w}\Bigr)\,w^{-2}\,\mathrm{d}w
+\int_{r^\alpha/\lambda}^{\infty}w^{-2}\,\mathrm{d}w\\
&= C r^{-\alpha}\int_{d_n}^{r^\alpha/\lambda}\frac{1+\log\!\bigl(r^\alpha/(\lambda w)\bigr)}{w}\,\mathrm{d}w
+ C r^{-\alpha}\\
&\le C r^{-\alpha}\Bigl(1+\log^2\!\frac{r^\alpha}{d_n}\Bigr).
\end{align*}
Hence
\begin{equation}\label{eq:pxy-bound-beta=1}
p(x,y)\ \le\ C\,\mathbf{1}_{\{x\in B_n\}}\mathbf{1}_{\{|x-y|>t_n\}}\,
|x-y|^{-\alpha}\Bigl(1+\log^2\!\frac{|x-y|^\alpha}{d_n}\Bigr).
\end{equation}

\smallskip\noindent
\emph{Case $\beta>1$.}
Since $\EE[W]<\infty$ and $1-e^{-u}\le u$, we have $\psi_\beta(u)\le u\,\EE[W]$ for all $u\ge 0$.
Therefore,
\begin{align*}
\int_{d_n}^{\infty}\psi_\beta\!\left(\frac{\lambda w}{r^\alpha}\right)\,\mathrm{d}F_W(w)
&\le \frac{\lambda\,\EE[W]}{r^\alpha}\int_{d_n}^{\infty} w\,\beta w^{-\beta-1}\,\mathrm{d}w
\ =\ C\,r^{-\alpha}\,d_n^{1-\beta},
\end{align*}
and thus
\begin{equation}\label{eq:pxy-bound-beta>1}
p(x,y)\ \le\ C\,\mathbf{1}_{\{x\in B_n\}}\mathbf{1}_{\{|x-y|>t_n\}}\,
|x-y|^{-\alpha}\,d_n^{1-\beta}.
\end{equation}

\smallskip
In particular, the bounds \eqref{eq:pxy-bound-beta<1}--\eqref{eq:pxy-bound-beta>1} may be summarized as
\[
p(x,y)\ \le\ C\,\mathbf{1}_{\{x\in B_n\}}\mathbf{1}_{\{|x-y|>t_n\}}\times
\begin{cases}
|x-y|^{-\alpha\beta}\Bigl(1+\log\!\dfrac{|x-y|^\alpha}{d_n}\Bigr), & 0<\beta<1,\\[2mm]
|x-y|^{-\alpha}\Bigl(1+\log^2\!\dfrac{|x-y|^\alpha}{d_n}\Bigr), & \beta=1,\\[2mm]
|x-y|^{-\alpha}\,d_n^{1-\beta}, & \beta>1.
\end{cases}
\]

    \emph{Step 3.} For $x,y \in \mathbb{R}$ we estimate $\mathbb{E}_{x,y}[ |U_{x,y} - V_{x,y}|]$. By definition we have
    \begin{align*}
        & |U_{x,y} - V_{x,y}|  = V_{x,y} - U_{x,y} \\
        & = \sum_{z\in \mathcal{P}\cap B^C_{t_n} (x)  } \mathbf{1}_{\{ {W_{x} > d_n}  \}}\mathbf{1}_{\{ {x\leftrightarrow z} \} } \mathbf{1}_{\{ x \in B_n \}} + \sum_{z\in \mathcal{P}\cap B^C_{t_n} (y)  } \mathbf{1}_{\{ {W_{y} > d_n}  \}}\mathbf{1}_{\{ {y\leftrightarrow z} \} } \mathbf{1}_{\{ y \in B_n \}} \\
        & \quad + \sum_{z\in \mathcal{P}\cap B_n (x)  } \mathbf{1}_{\{ {W_{z} > d_n}  \}}\mathbf{1}_{\{ {x\leftrightarrow z} \} } \mathbf{1}_{\{ x \in B^C_{t_n} (z) \}} + \sum_{z\in \mathcal{P}\cap B_n (y)  } \mathbf{1}_{\{ {W_{z} > d_n}  \}}\mathbf{1}_{\{ {y\leftrightarrow z} \} } \mathbf{1}_{\{ y \in B^C_{t_n} (z) \}} .
    \end{align*}
    Consequently, we have
    \begin{align*}
        \mathbb{E}_{x,y}[ |U_{x,y} - V_{x,y}|] & \leq \textnormal{I} + \textnormal{II} + \textnormal{III} + \textnormal{IV}
    \end{align*}
    with
     \begin{align*}
        \textnormal{I} & = \mathbf{1}_{\{ x \in B_n  \}}\int_{B^C_{t_n} (x) }\mathbb{P}_{x,y,z} \left( W_x > d_n, x\leftrightarrow z\right) \d z ,\\
        \textnormal{II} &= \mathbf{1}_{\{ y \in B_n  \}}\int_{B^C_{t_n} (y) }\mathbb{P}_{x,y,z} \left( W_y > d_n, y\leftrightarrow z\right) \d z ,\\
        \textnormal{III} & = \int_{B_n }\mathbb{P}_{x,y,z} \left( W_z > d_n, x\leftrightarrow z , |x-z|>t_n\right) \d z ,\\
        \textnormal{IV} &= \int_{B_n }\mathbb{P}_{x,y,z} \left( W_z > d_n, y\leftrightarrow z , |y-z|>t_n\right) \d z .
    \end{align*}
    Thus, making use of previous calculations we obtain
    $$ \mathbb{E}_{x,y}[ |U_{x,y} - V_{x,y}|] \leq w(x,y) $$
    with
  \[
w(x,y)
:=\int_{\RR^d}
\mathbf{1}_{\{x\in B_n\}}\mathbf{1}_{\{|z|>t_n\}}\,
q_\beta(|z|;d_n)\,\mathrm{d}z,
\]
where, for $r>0$,
\[
q_\beta(r;d_n)=
\begin{cases}
C\,r^{-\alpha\beta}\Bigl(1+\log\!\bigl(\frac{r^\alpha}{d_n}\bigr)\Bigr),
& 0<\beta<1,\\[1mm]
C\,r^{-\alpha}\Bigl(1+\log^2\!\bigl(\frac{r^\alpha}{d_n}\bigr)\Bigr),
& \beta=1,\\[1mm]
C\,d_n^{\,1-\beta}\,r^{-\alpha},
& \beta>1,
\end{cases}
\]
for a constant $C\in(0,\infty)$ independent of $n,x,y$.

\smallskip
This yields the regime-wise bounds
\[
w(x,y)\ \le\
\begin{cases}
C\displaystyle\int_{|z|>t_n}|z|^{-\alpha\beta}\Bigl(1+\log\!\bigl(\frac{|z|^\alpha}{d_n}\bigr)\Bigr)\,\mathrm{d}z,
& 0<\beta<1,\\[3mm]
C\displaystyle\int_{|z|>t_n}|z|^{-\alpha}\Bigl(1+\log^2\!\bigl(\frac{|z|^\alpha}{d_n}\bigr)\Bigr)\,\mathrm{d}z,
& \beta=1,\\[3mm]
C\,d_n^{\,1-\beta}\displaystyle\int_{|z|>t_n}|z|^{-\alpha}\,\mathrm{d}z,
& \beta>1.
\end{cases}
\]

\smallskip
Moreover, the corresponding tail integrals satisfy, for all $t_n\ge 1$ and $d_n\ge 1$,
\[
\int_{|z|>t_n}|z|^{-\alpha\beta}\Bigl(1+\log\!\bigl(\frac{|z|^\alpha}{d_n}\bigr)\Bigr)\,\mathrm{d}z
\ \le\
C\,t_n^{d-\alpha\beta}\Bigl(1+\log\!\bigl(\frac{t_n^\alpha}{d_n}\bigr)\Bigr),
\qquad (0<\beta<1,\ \alpha\beta>d),
\]
\[
\int_{|z|>t_n}|z|^{-\alpha}\Bigl(1+\log^2\!\bigl(\frac{|z|^\alpha}{d_n}\bigr)\Bigr)\,\mathrm{d}z
\ \le\
C\,t_n^{d-\alpha}\Bigl(1+\log^2\!\bigl(\frac{t_n^\alpha}{d_n}\bigr)\Bigr),
\qquad (\beta=1,\ \alpha>d),
\]
and
\[
\int_{|z|>t_n}|z|^{-\alpha}\,\mathrm{d}z
=\frac{d\omega_d}{\alpha-d}\,t_n^{d-\alpha},
\qquad (\beta>1,\ \alpha>d).
\]
Consequently,
\[
w(x,y)\ \le\
\begin{cases}
C\,t_n^{d-\alpha\beta}\Bigl(1+\log\!\bigl(\frac{t_n^\alpha}{d_n}\bigr)\Bigr), & 0<\beta<1,\\[1mm]
C\,t_n^{d-\alpha}\Bigl(1+\log^2\!\bigl(\frac{t_n^\alpha}{d_n}\bigr)\Bigr), & \beta=1,\\[1mm]
C\,d_n^{\,1-\beta}\,t_n^{d-\alpha}, & \beta>1.
\end{cases}
\]

    \emph{Step 4.} From \cite[Theorem 3.1]{Penrose2} we get the upper bound on the total variation distance
\begin{equation}\label{eq:tv-bound-start}
d_{\mathrm{TV}}(X_n,Z_n)\ \le\ C \int_{\RR^d}\int_{\RR^d} p(x,y)\,w(x,y)\,\mathrm{d}y\,\mathrm{d}x .
\end{equation}
Combining the estimates from the previous steps, we distinguish the three regimes of $\beta$.

\medskip
\noindent\textit{Case $0<\beta<1$.}
Using the bound (from Step~2)
\[
p(x,y)\ \le\ C\,\mathbf{1}_{\{x\in B_n\}}\mathbf{1}_{\{|x-y|>t_n\}}\,
|x-y|^{-\alpha\beta}\Bigl(1+\log\!\Bigl(\frac{|x-y|^\alpha}{d_n}\Bigr)\Bigr),
\]
and the definition/bound for $w(x,y)$,
\[
w(x,y)\ \le\ C\int_{|z|>t_n}|z|^{-\alpha\beta}\Bigl(1+\log\!\Bigl(\frac{|z|^\alpha}{d_n}\Bigr)\Bigr)\,\mathrm{d}z
\ \le\ C\,t_n^{d-\alpha\beta}\Bigl(1+\log\!\Bigl(\frac{t_n^\alpha}{d_n}\Bigr)\Bigr),
\]
we obtain from \eqref{eq:tv-bound-start}
\begin{align*}
d_{\mathrm{TV}}(X_n,Z_n)
&\le C\,t_n^{d-\alpha\beta}\Bigl(1+\log\!\Bigl(\frac{t_n^\alpha}{d_n}\Bigr)\Bigr)
\int_{B_n}\int_{|x-y|>t_n}|x-y|^{-\alpha\beta}
\Bigl(1+\log\!\Bigl(\frac{|x-y|^\alpha}{d_n}\Bigr)\Bigr)\,\mathrm{d}y\,\mathrm{d}x .
\end{align*}
For the remaining integral we use the same tail estimate again:
\[
\int_{|x-y|>t_n}|x-y|^{-\alpha\beta}
\Bigl(1+\log\!\Bigl(\frac{|x-y|^\alpha}{d_n}\Bigr)\Bigr)\,\mathrm{d}y
\le C\,t_n^{d-\alpha\beta}\Bigl(1+\log\!\Bigl(\frac{t_n^\alpha}{d_n}\Bigr)\Bigr),
\]
uniformly in $x\in\RR^d$. Hence
\begin{equation}\label{eq:tv-beta<1}
d_{\mathrm{TV}}(X_n,Z_n)
\ \le\
C\,(2n)^d\,t_n^{2(d-\alpha\beta)}
\Bigl(1+\log\!\Bigl(\frac{t_n^\alpha}{d_n}\Bigr)\Bigr)^{2}.
\end{equation}

\medskip
\noindent\textit{Case $\beta=1$.}
Proceeding analogously, we use
\[
p(x,y)\ \le\ C\,\mathbf{1}_{\{x\in B_n\}}\mathbf{1}_{\{|x-y|>t_n\}}\,
|x-y|^{-\alpha}\Bigl(1+\log^2\!\Bigl(\frac{|x-y|^\alpha}{d_n}\Bigr)\Bigr),
\]
and
\[
w(x,y)\ \le\ C\int_{|z|>t_n}|z|^{-\alpha}\Bigl(1+\log^2\!\Bigl(\frac{|z|^\alpha}{d_n}\Bigr)\Bigr)\,\mathrm{d}z
\ \le\ C\,t_n^{d-\alpha}\Bigl(1+\log^2\!\Bigl(\frac{t_n^\alpha}{d_n}\Bigr)\Bigr).
\]
Thus, as above,
\begin{equation}\label{eq:tv-beta=1}
d_{\mathrm{TV}}(X_n,Z_n)
\ \le\
C\,(2n)^d\,t_n^{2(d-\alpha)}
\Bigl(1+\log^2\!\Bigl(\frac{t_n^\alpha}{d_n}\Bigr)\Bigr)^{2}.
\end{equation}

\medskip
\noindent\textit{Case $\beta>1$.}
Here, we use
\[
p(x,y)\ \le\ C\,d_n^{1-\beta}\,\mathbf{1}_{\{x\in B_n\}}\mathbf{1}_{\{|x-y|>t_n\}}\,
|x-y|^{-\alpha},
\quad
w(x,y)\ \le\ C\,d_n^{1-\beta}\int_{|z|>t_n}|z|^{-\alpha}\,\mathrm{d}z
\le C\,d_n^{1-\beta}\,t_n^{d-\alpha}.
\]
Hence
\begin{align*}
d_{\mathrm{TV}}(X_n,Z_n)
&\le C\,d_n^{1-\beta}\,t_n^{d-\alpha}\int_{B_n}\int_{|x-y|>t_n} d_n^{1-\beta}|x-y|^{-\alpha}\,\mathrm{d}y\,\mathrm{d}x \\
&\le C\,d_n^{2(1-\beta)}\,t_n^{d-\alpha}\,(2n)^d\int_{|z|>t_n}|z|^{-\alpha}\,\mathrm{d}z \\
&\le C\,(2n)^d\,d_n^{2(1-\beta)}\,t_n^{2(d-\alpha)} .
\end{align*}
That is,
\begin{equation}\label{eq:tv-beta>1}
d_{\mathrm{TV}}(X_n,Z_n)
\ \le\
C\,(2n)^d\,t_n^{2(d-\alpha)}\,d_n^{2(1-\beta)} .
\end{equation}

Collecting \eqref{eq:tv-beta<1}--\eqref{eq:tv-beta>1}, we obtain the unified bound
\begin{equation*}\label{eq:tv-summary}
d_{\mathrm{TV}}(X_n,Z_n)\ \le\ C\,(2n)^d\,
\begin{cases}
t_n^{2(d-\alpha\beta)}
\Bigl(1+\log\!\bigl(t_n^\alpha/d_n\bigr)\Bigr)^{2},
& 0<\beta<1,\\[1mm]
t_n^{2(d-\alpha)}
\Bigl(1+\log^{2}\!\bigl(t_n^\alpha/d_n\bigr)\Bigr)^{2},
& \beta=1,\\[1mm]
t_n^{2(d-\alpha)}\,d_n^{2(1-\beta)},
& \beta>1,
\end{cases}
\end{equation*}
for all sufficiently large $n$, where $C\in(0,\infty)$ is a constant independent of $n$.

    It remains to prove the last claim
   \begin{equation*}%\label{eq:EXn-asymp}
\EE[X_n]\ \sim\  K (2n)^d\cdot
\begin{cases}
t_n^{\,d-\alpha\beta}\,\log\!\Bigl(\dfrac{t_n^\alpha}{d_n}\Bigr), & 0<\beta<1,\\[2.5mm]
t_n^{\,d-\alpha}\,\log^2\!\Bigl(\dfrac{t_n^\alpha}{d_n}\Bigr), & \beta=1,\\[2.5mm]
t_n^{\,d-\alpha}\,d_n^{\,1-\beta}, & \beta>1,
\end{cases}
\qquad n\to\infty,
\end{equation*}
with $K = K_{\alpha,\beta,d,\lambda}$ a constant independent of $n$. We have

\begin{align}
\EE[X_n]
&= \int_{B_n}\int_{B^c_{t_n}(x)} \PP\bigl(W_x>d_n,\ x\leftrightarrow y\bigr)\,\mathrm{d}y\,\mathrm{d}x \notag\\
&= \int_{B_n}\int_{B^c_{t_n}(x)} \int_{d_n}^{\infty}
\Bigl(1-\EE\Bigl[\exp\Bigl\{-\lambda \frac{wW_y}{|x-y|^\alpha}\Bigr\}\Bigr]\Bigr)\,\mathrm{d}F_W(w)\,\mathrm{d}y\,\mathrm{d}x \notag\\
&= (2n)^d\int_{B^c_{t_n}(0)} \int_{d_n}^{\infty}
\Bigl(1-\EE\Bigl[\exp\Bigl\{-\lambda \frac{wW}{|y|^\alpha}\Bigr\}\Bigr]\Bigr)\,\mathrm{d}F_W(w)\,\mathrm{d}y,
\label{eq:EXn-start}
\end{align}
where we used stationarity and wrote $W$ for a generic copy of the weight. Define for $\beta>0$
\[
\psi_\beta(x):=\EE\bigl[1-e^{-xW}\bigr],\qquad x\ge 0.
\]
Then \eqref{eq:EXn-start} can be rewritten as
\begin{equation}\label{eq:EXn-psi}
\EE[X_n]
=(2n)^d\int_{|y|>t_n}\int_{d_n}^{\infty}
\psi_\beta\!\Bigl(\frac{\lambda w}{|y|^\alpha}\Bigr)\,\mathrm{d}F_W(w)\,\mathrm{d}y.
\end{equation}

\smallskip
Recall that
\begin{equation}\label{eq:dn-over-tn}
\frac{d_n}{t_n^\alpha}\longrightarrow 0,
\qquad n\to\infty.
\end{equation}
Since $|y|>t_n$ implies $\lambda w/|y|^\alpha\le \lambda w/t_n^\alpha$, condition \eqref{eq:dn-over-tn}
ensures that the argument of $\psi_\beta$ is uniformly small on a dominant region of
$\{w\ge d_n,\ |y|>t_n\}$, and we may use the standard small--$x$ expansions
\begin{equation}\label{eq:psi-small-x}
\psi_\beta(x)\sim
\begin{cases}
\Gamma(1-\beta)\,x^\beta, & 0<\beta<1,\\[1mm]
x\log(1/x), & \beta=1,\\[1mm]
x\,\EE[W], & \beta>1,
\end{cases}
\qquad x\downarrow 0,
\end{equation}
together with the corresponding bounds $\psi_\beta(x)\le Cx^\beta$ ($0<\beta<1$),
$\psi_1(x)\le Cx(1+\log(1/x))$, and $\psi_\beta(x)\le Cx$ ($\beta>1$), for all $x\in(0,1]$.

\smallskip
Inserting \eqref{eq:psi-small-x} into \eqref{eq:EXn-psi} and evaluating the $y$--integrals in polar coordinates yields
the following regime--dependent asymptotic order:
\begin{equation*}%\label{eq:EXn-asymp}
\EE[X_n]\ \sim\ K (2n)^d\cdot
\begin{cases}
t_n^{\,d-\alpha\beta}\,\log\!\Bigl(\dfrac{t_n^\alpha}{d_n}\Bigr), & 0<\beta<1,\\[2.5mm]
t_n^{\,d-\alpha}\,\log^2\!\Bigl(\dfrac{t_n^\alpha}{d_n}\Bigr), & \beta=1,\\[2.5mm]
t_n^{\,d-\alpha}\,d_n^{\,1-\beta}, & \beta>1,
\end{cases}
\qquad n\to\infty,
\end{equation*}
where the constant $K=K_{\alpha,\beta,d,\lambda}$ is given by
\begin{equation}\label{eq:kappa-values}
K_{\alpha,\beta,d,\lambda} \,=\,
\begin{cases}
\dfrac{ d\omega_{d}\,\beta\,\Gamma(1-\beta)\,\lambda^\beta}{\alpha\beta-d},
& 0<\beta<1,\\[3mm]
\dfrac{ d \omega_{d}\,\lambda}{2(\alpha-d)},
& \beta=1,\\[3mm]
\dfrac{ d\omega_{d}}{\alpha-d}\,\lambda\Bigl(\dfrac{\beta}{\beta-1}\Bigr)^{\!2},
& \beta>1,
\end{cases}
\end{equation}
as claimed.

\subsection{Conclusuion}

%\begin{theorem}\label{th:POT1}
%    For some increasing sequence $(c_n)$, let $t_n=c_n^{\frac{1}{\alpha -d}}(1+\frac{t}{\alpha-d})$ and $d_n=\kappa^{-1}c_n$ with $\kappa=\lambda\beta d\omega_d(\beta-1)^{-2}(\alpha-d)^{-1}$. Assume that $t_n/\max(n,n^\frac{d}{\alpha-d})\rightarrow \infty$. Then, 
%\[\mathbb{P}_0\left(e^*_n>t_n\vert W_n^*>d_n\right)\rightarrow \left(1+\frac{t}{\alpha-d}\right)^{d-\alpha}.\]
%\end{theorem}
Now we complete the proof of Theorem \ref{thm:POT-hub}. The intuition of the proof is based on the observation in Remark \ref{re:intuitionmaxweights}. To make it precise, we write
    \begin{align*}
        & \mathbb{P}_0\left(e^*_n>t_n\vert W_n^*>d_n\right) \\
        & = \frac{\mathbb{P}_0\left(e^*_n>t_n, \exists ! \, x :  W_x>d_n\right)  }{\mathbb{P}_0\left( W_n^*>d_n\right)} +  \frac{\mathbb{P}_0\left(e^*_n>t_n, \exists   \textnormal{ at least two } x,y :  W_x>d_n, W_y>d_n\right)  }{\mathbb{P}_0\left( W_n^*>d_n\right)} .
    \end{align*}
    We will show that the second term in the latter sum is negligible. Indeed, using Lemma \ref{le:TailW*n} and the fact that $d_n^\beta n^{-d} \to \infty$, it is straightforward to see that
    \begin{align*}
        & \frac{\mathbb{P}_0\left(e^*_n>t_n, \exists   \textnormal{ at least two } x,y :  W_x>d_n, W_y>d_n\right)  }{\mathbb{P}_0\left( W_n^*>d_n\right)} \\
        & \leq \frac{\mathbb{P}_0\left( \exists   \textnormal{ at least two } x,y :  W_x>d_n, W_y>d_n\right)  }{\mathbb{P}_0\left( W_n^*>d_n\right)} \leq c \frac{n^{2d} d_n^{-2\beta}}{n^d d_n^{-\beta}} = c n^d d_n^{-\beta} \to 0.
    \end{align*}
    Thus, we focus on the first term and write
    \begin{align*}
        &  \frac{\mathbb{P}_0\left(e^*_n>t_n, \exists ! \, x :  W_x>d_n\right)  }{\mathbb{P}_0\left( W_n^*>d_n\right)} = \frac{\mathbb{P}_0\left(e^*_n>t_n, \exists ! \, x :  W_x>d_n\right)  }{\mathbb{P}_0\left(\exists ! \, x :  W_x>d_n\right)} \frac{\mathbb{P}_0\left( \exists ! \, x :  W_x>d_n\right)  }{\mathbb{P}_0\left( W_n^*>d_n\right)} .
    \end{align*}
    First, we consider the factor 
    $$\frac{\mathbb{P}_0\left( \exists ! \, x :  W_x>d_n\right)  }{\mathbb{P}_0\left( W_n^*>d_n\right)} = \frac{\mathbb{P}_0\left( \sum_{x\in \mathcal{P} \cap B_n}   \mathbf{1}_{\{W_x>d_n\}} = 1\right)  }{\mathbb{P}_0\left( W_n^*>d_n\right)} .$$
    It is easy to see that the random variable $\sum_{x\in \mathcal{P} \cap B_n}   \mathbf{1}_{\{W_x>d_n\}}$ approximates a Poisson distribution. More precisely, we have
    \begin{align*}
        \mathbb{P}_0\left( \sum_{x\in \mathcal{P} \cap B_n}   \mathbf{1}_{\{W_x>d_n\}} = 1\right) \sim (2n)^d d_n^{-\beta} \exp (- (2n)^d d_n^{-\beta}),
    \end{align*}
    so that along with Lemma \ref{le:TailW*n} we have
     \begin{align*}
            \frac{\mathbb{P}_0\left( \exists ! \, x :  W_x>d_n\right)  }{\mathbb{P}_0\left( W_n^*>d_n\right)} &\sim \frac{(2n)^d d_n^{-\beta} \exp (- (2n)^d d_n^{-\beta})}{1-\left(1-d_n^{-\beta}\right)\exp\left(-(2n)^dd_n^{-\beta}\right)} \\
            & \sim \frac{(2n)^d d_n^{-\beta} \exp (- (2n)^d d_n^{-\beta})}{(2n)^dd_n^{-\beta}+d_n^{-\beta}\exp\left(-(2n)^dd_n^{-\beta}\right)} \to 1,
    \end{align*}
    where we made use of the assumption $d_n^\beta n^{-d} \to \infty$ again. Thus, it remains to calculate the limit of
    $$\frac{\mathbb{P}_0\left(e^*_n>t_n, \exists ! \, x :  W_x>d_n\right)  }{\mathbb{P}_0\left(\exists ! \, x :  W_x>d_n\right)} ,$$
    which we will do by also applying Lemma \ref{le:jointedgesweights}. We write
    \begin{align*}
        &\frac{\mathbb{P}_0\left(e^*_n>t_n, \exists ! \, x :  W_x>d_n\right)  }{\mathbb{P}_0\left(\exists ! \, x :  W_x>d_n\right)} \\
        & =  \frac{\mathbb{P}_0\left( \exists ! \, x :  W_x>d_n, e^*_x>t_n\right) +  \mathbb{P}_0\left( \exists ! \, x :  W_x>d_n, e^*_x<t_n, \bar{e}_n(x)>t_n\right) }{\mathbb{P}_0\left(\exists ! \, x :  W_x>d_n\right)} ,
    \end{align*}
    where as before, for $x \in V \cap B_n$,
    \[
e_x^* \,:=\, \sup\bigl\{|y|:\ y\in V\setminus\{x\},\ x\leftrightarrow y\bigr\}, \quad \bar{e}_n(x) := \max_{ y\in V\setminus\{x\}} e_y^*
\]
    and
    \begin{align*}
         \frac{ \mathbb{P}_0\left( \exists ! \, x :  W_x>d_n, e^*_x<t_n, \bar{e}_n(x)>t_n\right) }{\mathbb{P}_0\left(\exists ! \, x :  W_x>d_n\right)} &\leq \frac{ \mathbb{P}_0\left( \exists  \, x :  W_x>d_n, e^*_x<t_n, \bar{e}_n(x)>t_n\right) }{\mathbb{P}_0\left(\exists ! \, x :  W_x>d_n\right)} \\
        & \leq   \frac{\mathbb{P}_0\left( \exists  \, x :  W_x>d_n,  \bar{e}_n(x)>t_n\right) }{\mathbb{P}_0\left(\exists ! \, x :  W_x>d_n\right)} \\
        & \leq c  \frac{ n^d d_n^{-\beta}\mathbb{P}_0\left(  \bar{e}_n(x)>t_n\right) }{\mathbb{P}_0\left(\exists ! \, x :  W_x>d_n\right)} \\
        & \leq c  \mathbb{P}_0\left(  \bar{e}_n(x)>t_n\right)  \leq c  \mathbb{P}_0\left(  e_n^*>t_n\right) \to 0.
    \end{align*}
    Thus, it only remains to consider
     \begin{align*}
        &   \frac{\mathbb{P}_0\left( \exists ! \, x :  W_x>d_n, e^*_x>t_n\right) }{\mathbb{P}_0\left(\exists ! \, x :  W_x>d_n\right)} \\
        & = \frac{\mathbb{P}_0\left( \exists ! \, x :  W_x>d_n, \exists ! \, y \in B_{t_n}^C (x) :  x \leftrightarrow y\right) }{\mathbb{P}_0\left(\exists ! \, x :  W_x>d_n\right)} \\
        & \quad + \frac{\mathbb{P}_0\left( \exists ! \, x :  W_x>d_n, \exists  \, \textnormal{ at least two }y_1, y_2 \in B_{t_n}^C (x) :  x \leftrightarrow y_1 , x \leftrightarrow y_2\right) }{\mathbb{P}_0\left(\exists ! \, x :  W_x>d_n\right)}\\
        & = \frac{\mathbb{P}_0\left( X_n = 1\right) }{\mathbb{P}_0\left(\exists ! \, x :  W_x>d_n\right)} \\
        & \quad + \frac{\mathbb{P}_0\left( \exists ! \, x :  W_x>d_n, \exists  \, \textnormal{ at least two }y_1, y_2 \in B_{t_n}^C (x) :  x \leftrightarrow y_1 , x \leftrightarrow y_2\right) }{\mathbb{P}_0\left(\exists ! \, x :  W_x>d_n\right)} ,
    \end{align*}
    where $X_n$ is as defined in Lemma \ref{le:jointedgesweights}. We  estimate the second term in the latter sum. Using Lemma \ref{le:jointedgesweights}, we obtain:
    \begin{align*}
        & \frac{\mathbb{P}_0\left( \exists ! \, x :  W_x>d_n, \exists  \, \textnormal{ at least two }y_1, y_2 \in B_{t_n}^C (x) :  x \leftrightarrow y_1 , x \leftrightarrow y_2\right) }{\mathbb{P}_0\left(\exists ! \, x :  W_x>d_n\right)} \\
        & \leq \frac{\mathbb{P}_0\left( X_n \geq 2\right) }{\mathbb{P}_0\left(\exists ! \, x :  W_x>d_n\right)} = \frac{1-\mathbb{P}_0\left( X_n = 0\right) -  \mathbb{P}_0\left( X_n = 1\right)}{\mathbb{P}_0\left(\exists ! \, x :  W_x>d_n\right)} \\
        & \sim \frac{1-\exp (-\mathbb{E}[X_n]) -  \mathbb{E}[X_n]\exp (-\mathbb{E}[X_n])}{(2n)^d d_n^{-\beta} \exp (- (2n)^d d_n^{-\beta})} \\
        & \leq \frac{\mathbb{E}[X_n] -  \mathbb{E}[X_n]\exp (-\mathbb{E}[X_n])}{(2n)^d d_n^{-\beta} \exp (- (2n)^d d_n^{-\beta})} \\
        %&\sim \frac{\kappa (2n)^d t_n^{d-\alpha} d_n^{-\beta +1} -  \kappa (2n)^d t_n^{d-\alpha} d_n^{-\beta +1}\exp (-\kappa (2n)^d t_n^{d-\alpha} d_n^{-\beta +1})}{(2n)^d d_n^{-\beta} \exp (- (2n)^d d_n^{-\beta})} \\
       % & = \frac{\kappa  t_n^{d-\alpha} d_n -  \kappa  t_n^{d-\alpha} d_n\exp (-\kappa (2n)^d t_n^{d-\alpha} d_n^{-\beta +1})}{ \exp (- (2n)^d d_n^{-\beta})} \\
        %& \to \left(1+\frac{t}{\alpha-d}\right)^{d-\alpha} - \left(1+\frac{t}{\alpha-d}\right)^{d-\alpha} = 0,
    \end{align*}
 Set $A_n:=(2n)^d d_n^{-\beta}$ and $\mu_n:=\mathbb{E}[X_n]$. Then
\[
\frac{\mathbb{E}[X_n]-\mathbb{E}[X_n]e^{-\mathbb{E}[X_n]}}{(2n)^d d_n^{-\beta}e^{-(2n)^d d_n^{-\beta}}}
=\frac{\mu_n(1-e^{-\mu_n})}{A_n e^{-A_n}}.
\]
By the definition of the sequences $t_n$ and $d_n$, we have  $A_n\to 0$, $\mu_n\to 0$, and $\mu_n^2/A_n\to 0$. Since $e^{-A_n}\to 1$ and
$1-e^{-\mu_n}\sim \mu_n$ as $\mu_n\downarrow 0$, we obtain
\[
\frac{\mu_n(1-e^{-\mu_n})}{A_n e^{-A_n}}
\sim \frac{\mu_n^2}{A_n}\longrightarrow 0.
\]
    
    Thus, it only remains to consider the limit of 
    $$\frac{\mathbb{P}_0\left( X_n = 1\right) }{\mathbb{P}_0\left(\exists ! \, x :  W_x>d_n\right)} .$$
    Using Lemma \ref{le:jointedgesweights} as in the previous step, we obtain:
    \begin{align*}
        \frac{\mathbb{P}_0\left( X_n = 1\right) }{\mathbb{P}_0\left(\exists ! \, x :  W_x>d_n\right)} & \sim  \frac{\mathbb{E}[X_n]\exp (-\mathbb{E}[X_n]) }{(2n)^d d_n^{-\beta} \exp (- (2n)^d d_n^{-\beta})} \\
       &  \sim \frac{\mathbb{E}[X_n] }{(2n)^d d_n^{-\beta} }
       % & \sim \frac{\kappa (2n)^d t_n^{d-\alpha} d_n^{-\beta +1}\exp (-\kappa (2n)^d t_n^{d-\alpha} d_n^{-\beta +1}) }{(2n)^d d_n^{-\beta} \exp (- (2n)^d d_n^{-\beta})} \\
        %& = \frac{\kappa  t_n^{d-\alpha} d_n\exp (-\kappa (2n)^d t_n^{d-\alpha} d_n^{-\beta +1}) }{ \exp (- (2n)^d d_n^{-\beta})}
        \to \left(1+\frac{t}{\theta}\right)^{-\theta} ,
    \end{align*}
    as before, where we used Lemma \ref{le:jointedgesweights} twice, the exact definition and assumption on the sequences $t_n$ and $d_n$ given in Theorem \ref{thm:POT-hub}. The proof is complete.
 \medskip

\paragraph{\textbf{Acknowledgements.}}
Part of this work was carried out during a research stay at CIRM (Centre International de Rencontres Mathématiques). The authors gratefully acknowledge CIRM for its hospitality and excellent working conditions. The IMB receives support from the EIPHI Graduate School (contract ANR-17-EURE-0002).

\bibliographystyle{amsplain}
\bibliography{lit}

\providecommand{\bysame}{\leavevmode\hbox to3em{\hrulefill}\thinspace}
\providecommand{\MR}{\relax\ifhmode\unskip\space\fi MR }
\providecommand{\MRhref}[2]{%
  \href{http://www.ams.org/mathscinet-getitem?mr=#1}{#2}
}
\providecommand{\href}[2]{#2}
\begin{thebibliography}{1}

\bibitem{barbour}
A.~D. Barbour, L.~Holst, and S.~Janson, \emph{Poisson approximation}, vol.~2,
  The Clarendon Press Oxford University Press, 1992.

\bibitem{deijfen2013scale}
M.~Deijfen, R.~Van~der Hofstad, and G.~Hooghiemstra, \emph{Scale-free
  percolation}, Annales de l'IHP Probabilit{\'e}s et statistiques, vol.~49,
  2013, pp.~817--838.

\bibitem{deprez2019scale}
P.~Deprez and M.~W{\"u}thrich, \emph{Scale-free percolation in continuum
  space}, Communications in Mathematics and Statistics \textbf{7} (2019),
  no.~3, 269--308.

\bibitem{iyer2}
S.~K. Iyer and S.~K. Jhawar, \emph{Poisson approximation and connectivity in a
  scale-free random connection model}, Electronic Journal of Probability
  \textbf{26} (2021), 1--23.

\bibitem{Penrose2}
M.~D. Penrose, \emph{{Inhomogeneous random graphs, isolated vertices, and
  Poisson approximation}}, Journal of Applied Probability \textbf{55} (2018),
  no.~1, 112--136.

\bibitem{rs1}
A.~Rousselle and E.~S{\"o}nmez, \emph{The longest edge in discrete and
  continuous long-range percolation}, Extremes \textbf{27} (2024), no.~4,
  673--703.

\bibitem{rs2}
\bysame, \emph{The longest edge of the one-dimensional soft random geometric
  graph with boundaries}, Stochastic Models \textbf{40} (2024), no.~2,
  399--416.

\bibitem{SW}
R.~Schneider and W.~Weil, \emph{Stochastic and integral geometry}, Probability
  and its Applications (New York), Springer-Verlag, Berlin, 2008. \MR{2455326}

\end{thebibliography}

\end{document}